\documentclass[openright,a4paper,american,11pt]{amsart}
\usepackage[latin1]{inputenc}
\usepackage{graphicx}
\usepackage{amsmath}
\usepackage{amssymb}
\usepackage[english]{babel}
\usepackage{caption}
\usepackage{subcaption}
\usepackage{amsfonts}
\usepackage{amscd}
\usepackage{wrapfig}
\usepackage{boxedminipage}
\usepackage{hyperref}
\usepackage{tikz}

\hypersetup{pdfstartpage=1, pdftitle="The Newton polygon of a planar singular curve and its subdivision", pdfauthor="Nikita
  Kalinin", pdfstartview={FitH}, pagebackref=true, pdfkeywords={tropical singularities, lattice width}}

\newcommand{\ZZ}{{\mathbb Z}}
\newcommand{\KK}{{\mathbb K}}
\newcommand{\CC}{{\mathbb C}}
\newcommand{\RR}{{\mathbb R}}
\newcommand{\QQ}{{\mathbb Q}}
\newcommand{\FF}{{\mathbb F}}
\newcommand{\TT}{{\mathbb T}}
\newcommand{\E}{{\mathfrak E}}
\newcommand{\A}{{\mathcal A}}

\newcommand{\val}{\mathrm{val}}
\newcommand{\Val}{\mathrm{Val}}
\newcommand{\Trop}{\mathrm{Trop}}
\newcommand{\Inf}{\mathfrak{Infl}}

\newcommand{\I}{\mathfrak{I}}
\newcommand{\conv}{\mathrm{ConvHull}}
\newcommand{\deff}{\mathrm{def}}
\newcommand{\area}{\mathrm{area}}
\newcommand{\dirr}{P(\ZZ^2)}
\DeclareMathOperator{\trop}{trop}

\newtheorem{theo}{Theorem}
\newtheorem{thm}{Theorem}[section]
\newtheorem{prop}[thm]{Proposition}
\newtheorem{lemma}[thm]{Lemma}
\newtheorem{cor}[thm]{Corollary}
         {\theoremstyle{definition}
\newtheorem{rem}[thm]{Remark}}
         {\theoremstyle{definition}
\newtheorem{ex}[thm]{Example}}
 {\theoremstyle{definition}
\newtheorem{defi}[thm]{Definition}}

\newtheorem*{th_old}{Theorem}
\begin{document}

\title{THE NEWTON POLYGON OF A PLANAR SINGULAR CURVE AND ITS SUBDIVISION} 

\date{\today}

\author[N. Kalinin]{Nikita Kalinin}

\address{Universit\'e de Gen\`eve, Section de
  Math\'ematiques, Route de Drize 7, Villa Battelle, 1227 Carouge, Switzerland} 

\address{St. Petersburg Department of the Steklov Mathematical Institute, Russian Academy of Sciences, Fontanka 27,
St. Petersburg, 191023 Russia.}

\email{Nikita.Kalinin\{dog\}unige.ch, nikaanspb\{dog\}gmail.com}

\begin{abstract}
Let a planar algebraic curve $C$ be defined over a valuation field by an equation
$F(x,y)=0$. Valuations of the coefficients of $F$ define a subdivision of the
Newton polygon $\Delta$ of the curve $C$.  
 
If a given point $p$ is of multiplicity $m$ on $C$, then the coefficients of
$F$ are subject to certain linear constraints. These constraints can be visualized
in the above subdivision of $\Delta$. Namely, we find a distinguished
collection of faces of the above subdivision, with total area at least
$\frac{3}{8}m^2$. The union of these faces can be considered to be the ``region of
influence'' of the singular point $p$ in the subdivision of $\Delta$. We also discuss three different definitions of a tropical point of multiplicity $m$.

\end{abstract}


\maketitle

\section{Introduction}

Fix a non-empty finite subset $\A\subset \ZZ^2$ and any valuation
field $\KK$. We consider a curve $C$ given by an
equation $F(x,y)=0$, where \begin{equation}
\label{eq_curve}
F(x,y)=\sum\limits_{(i,j)\in \A
}a_{ij}x^iy^j, \ a_{ij}\in \KK^*.
\end{equation} 

Suppose that we know
only the valuations of the coefficients of the polynomial
$F(x,y)$. Is it possible to extract
any meaningful information from this knowledge? Unexpectedly, many
geometric properties of $C$ are visible from such a viewpoint.

The {\it Newton polygon} $\Delta=\Delta(\A)$ of the curve $C$ is the convex hull of $\A$ in $\RR^2$. The {\it extended
  Newton polyhedron} $\widetilde{\A}$ of the curve $C$ is the convex hull
of the set
$\{((i,j),s)\in\RR^2\times \RR| (i,j)\in \A, s\leq
\val(a_{ij})\}$. Projection of all the faces of
$\widetilde{\A}$ along $\RR$ induces a {\it subdivision} of
$\Delta$. Note that the valuations of the coefficients of $F$
completely determine $\widetilde{\A}$ and this subdivision of $\Delta$.

A point $p$ is
{\it of multiplicity $m$} (or is an $m$-fold point) on $C$ 
if the lowest term in the Taylor
expansion of $F$ at $p$ has degree $m$. Nagata's conjecture
proposes the estimate $d\geq m\sqrt n$ for the minimal degree $d$ of a curve which has $n>9$
points of multiplicity $m$ in general position. Motivated by this
conjecture, we study the following
question: how do the points of multiplicity $m$ on $C$ 
influence the subdivision of $\Delta$? This paper is devoted
to the case of one $m$-fold point, whereas \cite{2013arXiv1310.6684K} concerns the
case of several $m$-fold points.

By definition, the non-Archimedean amoeba of $C$ is
$\Val(C)=\{(\val(x),\val(y))|(x,y)\in C\}.$ Also, we define the
tropical curve $\Trop(C)$ as the set of non-smooth points of the
function $\max\limits_{(i,j)\in \A}(iX+jY+\val(a_{ij}))$. It is known
that $\Val(C)\subset \Trop(C)$. Furthermore, $\Trop(C)$ is a graph which is combinatorially
dual to the subdivision of $\Delta$ (described above). In particular, 
each vertex $V$ of $\Trop(C)$ corresponds to a face $d(V)$
of this subdivision of $\Delta$. 

Fix a point $p=(p_1,p_2)\in (\KK^*)^2$. Define $P=\Val(p)=(\val(p_1),\val(p_2))$. We consider a curve $C$ given by \eqref{eq_curve} such
that $p$ is of multiplicity $m$ on $C$. In such a case, the coefficients $a_{ij}$ of
$C$ satisfy a certain set of $\frac{m(m+1)}{2}$ linear
constraints.  In turn, the constraints for the numbers
$\val(a_{ij})$ manifest themselves via the fact that the subdivision of $\Delta$ enjoys very special properties.

\begin{figure}[htb]
\begin{center}
\begin{subfigure}[b]{0.3\textwidth}
\begin{tikzpicture}[scale=0.3]
\draw [very thin, gray] (0,-1) grid (15,11);
\draw[very thick] (1,10)--(10,10)--(12,9)--(14,6)--(14,0)--(2,1)--cycle;
\draw[very thick]
(2,6)--(3,6)--(4,5)--(6,8)--(8,9)--(10,8)--(12,7)--(13,6)--(12,6)--(10,5)--(8,3)--(6,4)--(4,3)--(3,5)--(2,6);

\draw[very thick] (3,6)--(3,5);
\draw[very thick] (4,5)--(4,3);
\draw[very thick] (6,8)--(6,4);
\draw[very thick] (8,9)--(8,3);
\draw[very thick] (10,8)--(10,5);
\draw[very thick] (12,7)--(12,6);
\end{tikzpicture}
\end{subfigure}
\quad\quad\quad
\begin{subfigure}[b]{0.3\textwidth}
\begin{tikzpicture}[scale=0.3]
\draw [very thin, gray] (0,0) grid (16,12);
\draw[very thick] (2,0)--(13,0)--(16,2)--(16,9)--(8,12)--(1,8)--cycle;
\draw[very thick]
(3,7)--(4,7)--(8,11)--(8,12)--(10,11)--(11,10)--(12,8)--(13,8)--(16,5)--(13,5)--(12,4)--(12,2)--(10,2)--(10,0)--(9,0)--(7,2)--(5,5)--cycle;

\draw[very thick] (4,7)--(4,6);
\draw[very thick] (5,8)--(5,5);
\draw[very thick] (8,11)--(10,11);
\draw[very thick] (7,10)--(11,10);
\draw[very thick] (12,8)--(12,4);
\draw[very thick] (13,8)--(13,5);
\draw[very thick] (14,7)--(14,5);
\draw[very thick] (10,2)--(12,4);
\draw[very thick] (11,2)--(12,3);
\draw[very thick] (7,2)--(10,2);
\draw[very thick] (8,1)--(10,1);

\draw[very thick] (5,8)--(6,10)--(7,10);
\draw[very thick] (11,10)--(12,10)--(12,8);
\draw[very thick] (5,5)--(4,3)--(7,2);
\end{tikzpicture}
\end{subfigure}
\end{center}
\begin{center}
\begin{subfigure}[h]{0.3\textwidth}
\begin{tikzpicture}[scale=0.3]
\path (0,0);
\begin{scope}[shift={(0,-5.5)}]
\draw (0,0) node {$\bullet$};
\draw (0,0) node[above] {$P$};

\draw[very thick] (-7,0)--(5,0);
\draw[very thick] (2,2)--(1,0)--(2,-1);
\draw[very thick] (4,2)--(3,0)--(4,-2);
\draw[very thick] (6,1)--(5,0)--(5,-1);
\draw[very thick] (-3,2)--(-2,0)--(-3,-2);
\draw[very thick]  (-4.75,0.5)--(-4,0)--(-3.5,-2);
\draw[very thick]  (-5.5,0.5)--(-6,0)--(-6,-0.5);
\draw[very thick] (-7,1)--(-7,0)--(-7.5,-0.5);
\end{scope}
\end{tikzpicture}
\vspace{23pt}
\caption{if $\Val(p)$ is not a vertex}
\label{governorship_1}
\end{subfigure}
\quad\quad\quad
\begin{subfigure}[ht]{0.3\textwidth}
\begin{tikzpicture}[scale=0.3]
\path (-5,0);
\begin{scope}[shift={(26mm,0)}]
\draw (0,0) node {$\bullet$};
\draw (0,0) node[above right] {$P$};
\draw[very thick] (0,0)--(0,2)--(1,3);
\draw[very thick] (-1,3)--(0,2)--(0,4)--(-1,4);
\draw[very thick] (0,4)--(0.5,5);
\draw[very thick] (0,0)--(2,1)--(3,1);
\draw[very thick]  (2,1)--(2,2);
\draw[very thick]  (0,0)--(-2,2)--(-2,3);
\draw[very thick]  (-2,2)--(-4,3);
\draw[very thick] (0,0)--(3.5,0)--(3.5,0.5);
\draw[very thick] (4,-0.5)--(3.5,0)--(5,0)--(5,-1);
\draw[very thick]  (6,1)--(5,0)--(6,0)--(7,1);
\draw[very thick] (6,0)--(6,-1);
\draw[very thick]  (0,0)--(-3,0)--(-4,1);
\draw[very thick] (-4,-1)--(-3,0)--(-5,0)--(-5,1);
\draw[very thick] (-5,0)--(-6,-1);
\draw[very thick] (0,0)--(-3,-2)--(-4,-1.5);
\draw[very thick] (-3,-2)--(-4,-5);
\draw[very thick]  (0,0)--(0,-2.5)--(1,-2.5);
\draw[very thick] (-1,-3.5)--(0,-2.5)--(0,-4)--(1,-4);
\draw[very thick]  (-1,-5)--(0,-4)--(0,-5);
\draw[very thick]  (0,0)--(1.5,-1.5)--(2.5,-1.5);
\draw[very thick] (1.5,-2.5)--(1.5,-1.5)--(2.5,-2.5)--(3.5,-2.5);
\draw[very thick] (2.5,-3.5)--(2.5,-2.5);
\end{scope}
\end{tikzpicture}
\caption{if $\Val(p)$ is a vertex}
\label{governorship_2}
\end{subfigure}
\end{center}

\caption{If $P$ is not a vertex of $\Trop(C)$ (left column), then the
  collection $\I(P)$ of vertices consists of all the vertices of
  $\Trop(C)$ lying on the extension of the edge through $P$.  If $P$ is a vertex of $\Trop(C)$ (right column), then we take the vertices on the extensions of all the edges through $P$. In each case the
  corresponding set of faces of the subdivision of
  $\Delta$, the ``region of influence'' of $P$, is drawn at the top. The sum of the
  areas of the faces in \eqref{eq_estimate} is at least $\frac{1}{2}m^2$
  in (A) and at least $\frac{3}{8}m^2$ in (B).}
\label{governorship}

\end{figure}
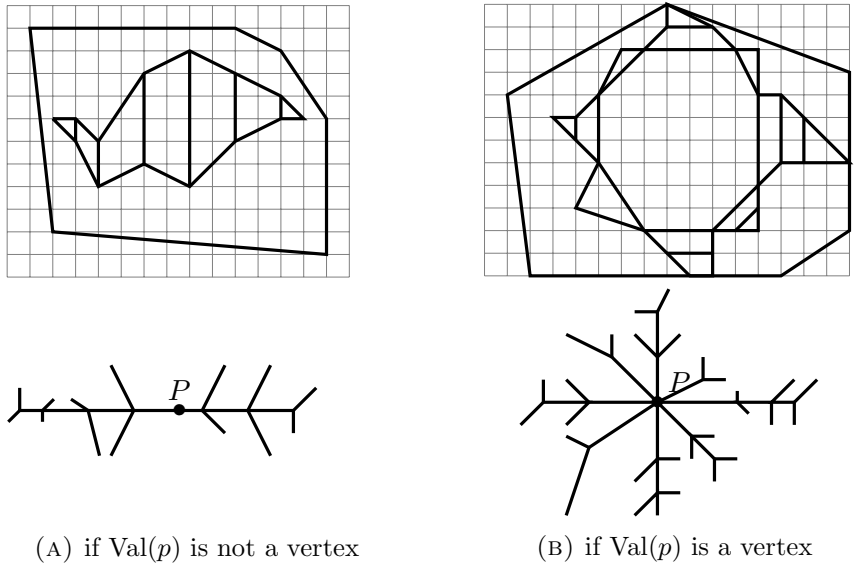

In particular, there is a certain collection $\I(P)$ of vertices of
$\Trop(C)$ (Figure~\ref{governorship}, lower row). We estimate the total area of the faces in the
subdivision of $\Delta$ dual to the vertices in $\I(P)$
(Figure~\ref{governorship}, upper row). Namely, if the minimal lattice width of $\Delta$ is at least $m$, then the following
inequality holds: 
\begin{equation}
\label{eq_estimate}
\sum_{V\in \I(P)}\area(d(V))\geq cm^2.
\end{equation}
 If $P$ is not a vertex of
$\Trop(C)$, then \eqref{eq_estimate} holds with $c=\frac{1}{2}$; if $P$ is a
vertex of $\Trop(C)$, then \eqref{eq_estimate} holds with
$c=\frac{3}{8}$, see Lemma \ref{lemma_implication}, Theorems \ref{exertion_edges},\ref{exertion_vertices} in Section \ref{formulation}  for more details.
\begin{rem}
\label{motivation}
Let us fix points $p_1,p_2,\dots, p_n$ in general position. Suppose that $C$ passes through them. 
In \cite{2013arXiv1310.6684K}, we prove that in this case each vertex
of $\Trop(C)$ belongs to at most two sets $\I(P_i)$,
i.e., for indices $i_1<i_2<i_3$ we have $\I(P_{i_1})\cap
\I(P_{i_2})\cap\I(P_{i_3})=\varnothing$. 
\end{rem}
\begin{defi}[\cite{hyper,markwig}]
\label{def_extrinsic}The multiplicity of a point $P$ on a tropical
curve $H$ is at least $m$ {\it in the $\KK$-extrinsic sense} if there
exists an algebraic curve $H'\subset (\KK^*)^2$ and a point $p\in H'$ of
multiplicity $m$ such that $\Trop(H')=H,\Val(p)=P$.
\end{defi}

This definition is extrinsic because it involves other objects
besides $H$. We find new necessary
intrinsic conditions (in terms of the
subdivision of $\Delta$) for the presence of
an $m$-fold point on $C$. We give two other definitions (Def.~\ref{def_intrinsic}, Def.~\ref{def_intermediate}) of a tropical
singular point and compare them in Section~\ref{tropical_multiplicity}.

Previous research in this direction has been carried out for $m=2$ in \cite{markwig,markwig2}, for
inflection points in \cite{brugalle}, and for cusps in \cite{ganor}. 
Lifting of tropical singular points to the usual singular points is discussed in
\cite{patch}. In \cite{hyper2,discr,hyper} the matroid $M$ associated with
the aforementioned linear constraints on $a_{ij}$ is studied; see also
Remark \ref{euler}. 

{\bf Acknowledgments.} I thank M.~Karev, S.~Lanzat, G.~Mikhalkin, Ch.~Sia, and two referees for
help with editing this text. Research is supported in part by grant 159240 of the Swiss National Science
Foundation as well as by the National Center of Competence in Research
SwissMAP of the Swiss National Science Foundation.

\section{Preliminaries}
\label{preliminaries}

\subsection{Tropical geometry and valuation fields} Let $\TT$ denote $\RR\cup \{-\infty\}$. $\TT$ is usually called {\it
  the tropical semi-ring}. Let $\KK$ be any valuation field, i.e., a
field equipped with a {\it valuation map} $\val: \KK \to \TT$, where
this map $\val$ possesses the following properties:

\begin{itemize}
\item $\val(ab)=\val(a)+\val(b)$,
\item $\val(a+b)\leq \max(\val(a),\val(b))$,
\item $\val(0)=-\infty$.
\end{itemize}

\begin{ex}
Let $\FF$ be an arbitrary (possibly finite) field. An example of a
valuation field is the field $\FF \{\{t\}\}$ of
generalized Puiseux series. Namely, 
$$\FF \{\{t\}\} = \bigg\{\sum_{\alpha\in I}c_\alpha t^\alpha |
c_\alpha\in \FF,I\subset \RR\bigg\},$$ where $t$ is a formal variable and $I$ is
a well-ordered set, i.e., each of its nonempty subsets has a least
element. The valuation map $\val: \KK \to \TT$ is defined by the rule $$\val\Big(\sum\limits_{\alpha\in
  I}c_\alpha t^\alpha\Big):=-\min\limits_{\alpha\in I}\{\alpha|c_\alpha\ne
0\}, \val(0):=-\infty.$$  
\end{ex}

Different constructions of Puiseux series and their properties are
listed in \cite{puiseux, puiseaux2}. 

\begin{rem}
\label{rem_max} It follows from the axioms of the valuation map that
if $a_1+a_2+\dots+a_n=0, a_i\in\KK^*$, then the maximum among
$\val(a_i),i=1,\dots,n$ is attained at least twice.
\end{rem}

\begin{ex}Suppose that $\KK=\CC \{ \{t\}\}$ and
all the coefficients $a_{ij}\in\KK^*$ in \eqref{eq_curve} are
convergent series in $t$ for $t$ close
to zero. Then, specializing $t$ to be $t_k\in\CC$ close to zero, we obtain a family of complex
curves $C_{t_k}$ defined by the equations $\sum_{(i,j)\in \A}
a_{ij}(t_k)x^iy^j=0$. Note that the valuation $\val\big(\sum\limits_{\alpha\in
  I}c_\alpha t^\alpha\big)=-\min\limits_{\alpha\in
  I}\{\alpha|c_\alpha \ne 0\}$ is a measure of the
asymptotic behavior of $a_{ij}$ as $t_k$ tends to 0, i.e.,
$a_{ij}(t_k) \sim  t_k^{-\val(a_{ij})}$.
\end{ex}

 The
combinatorics of the extended Newton polyhedron reflects some
asymptotically visible properties of a generic member of the family
$\{C_{t_k}\}$. In such a way, real algebraic curves with a prescribed topology can
be constructed; see Viro's patchworking method.

\begin{defi}[\cite{kapranov}] 
{\it The non-Archimedean amoeba} $\Val(C)\subset \TT^2$ of an algebraic curve $C\subset \KK^2$ is the image
 of $C$ under the map $\val$ applied coordinate-wise.  
\end{defi}

Now we recall some basic notions of tropical geometry.

\begin{defi}
For the given $F(x,y)=\sum\limits_{(i,j)\in\A} a_{ij}x^iy^j$, we define 
\begin{equation}
\label{eq_tropical}
\Trop(F)(X,Y)=\max\limits_{(i,j)\in \A}(iX+jY+\val(a_{ij})).
\end{equation}
\end{defi}

We use the letters $x,y$ for variables in $\KK$, and we use $X,Y$ for
the corresponding variables in
$\TT$.

Fix a finite subset $\A\subset \ZZ^2$. Let us consider a curve $C$ given by \eqref{eq_curve}.

\begin{defi}
\label{def_trop}
Let $\Trop(C)\subset \TT^2$ be the set of points where $\Trop(F)$ is not smooth,
that is, the set of points where the maximum in \eqref{eq_tropical} is attained at least twice.
\end{defi}

It is clear that $\Trop(C)$ is a planar graph, whose edges are straight.

\begin{rem}
We have $\Val(C)\subset\Trop(C)$ because if $F(x,y)=0$, then the
maximum among $\val(a_{ij}x^iy^j)$ must be
attained at least twice (Remark \ref{rem_max}). If  $\KK$ is algebraically closed and the image of $\val$ contains $\QQ$, then $\overline{\Val(C)}=\Trop(C)$
(c.f. \cite{kapranov}, Theorem 2.1.1). 
\end{rem}

To the curve $C$, we associate a subdivision of its Newton
polygon $\Delta = \conv(\A)$ by the following procedure. Consider
{\it the extended Newton
polyhedron} (\cite{kapranov}) $$\widetilde{\A}=\conv\Big(\bigcup
\{(i,j,x)|(i,j)\in \A, x\leq
\val(a_{ij})\}\Big)\subset \RR^3.$$ The projection of the edges of $\widetilde{\A}$ to the first two
coordinates gives us a subdivision of $\Delta$. Hence the curve $C$ produces the tropical
curve $\Trop(C)$ and the subdivision of $\Delta$.

\begin{prop}
\label{trop_prop}
This subdivision is dual to $\Trop(C)$ in the following sense (see Example~\ref{secondexample})
: 
\begin{itemize}
\item each vertex $Q$ of
$\Trop(C)$ corresponds to some face $d(Q)$ of the subdivision of $\Delta$; 
\item each edge $E$ of
$\Trop(C)$ corresponds to some edge $d(E)$ in the subdivision of $\Delta$, and the direction
of the edge $d(E)$ is perpendicular to the direction of $E$;
\item if a vertex $Q\in\Trop(C)$ is an
end of an edge $E\subset\Trop(C)$, then $d(Q)$ contains $d(E)$;
\item each vertex of
$\widetilde{A}$ corresponds to a connected component of
$\TT^2\setminus\Trop(C)$.
\end{itemize}
\end{prop}
\begin{proof}
This proposition follows from Def.~\ref{def_trop}.
\end{proof}

See Figure \ref{example3} for an example of the above duality. Also,
parts of tropical curves and the corresponding parts of the dual subdivisions are shown in Figure \ref{governorship}.

\begin{defi}
Suppose that $\Trop(F)$ is equal to $i_1X+j_1Y+\val(a_{i_1j_1})$ on one side of
an edge $E\subset \Trop(C)$ and to $i_2X+j_2Y+\val(a_{i_2j_2})$ on the other side of
$E$. Therefore $E$ is locally defined by the equation
$(i_1-i_2)X+(j_1-j_2)Y+(\val(a_{i_1j_1})-\val(a_{i_2j_2}))=0$. In this case the
endpoints of $d(E)$ are $(i_1,j_1),(i_2,j_2)$, and, by definition, the
{\it weight} of $E$
is equal to the lattice length of $d(E)$, which is $\mathrm{gcd}(i_1-i_2,j_1-j_2)$ by
definition.
\end{defi}

\begin{ex}
\label{secondexample}
Consider a curve $C'$ defined by the equation $G(x,y)=0$, where \begin{align*}G(x,y)&=t^{-3} xy^3 -
(3t^{-3}+t^{-2})xy^2+ (3t^{-3}+2t^{-2}-2t^{-1})xy -
(t^{-3}+t^{-2}-2t^{-1}-3t^2)x+\\&+t^{-2}x^2y^2
-(2t^{-2}-t^{-1})x^2y+(t^{-2}-t^{-1}-3t^2)x^2+t^{-1}y -(t^{-1}+t^2)
+t^2x^3.\end{align*}
\begin{figure} [htbp]
\begin{tikzpicture}
[y= {(0.2cm,1cm)}, z={(0cm,0.5cm)}, x={(2cm,0cm)},scale=0.5]
\draw (0,0,1)--(0,0,-3);
\draw (0,1,1)--(0,1,-3);
\draw (1,3,3)--(1,3,-3);
\draw (1,2,3)--(1,2,-3);
\draw (1,1,3)--(1,1,-3);
\draw (1,0,3)--(1,0,-3);
\draw (2,0,2)--(2,0,-3);
\draw (2,1,2)--(2,1,-3);
\draw (2,2,2)--(2,2,-3);
\draw (3,0,-2)--(3,0,-3);

\filldraw[gray](0,0,1)--(1,0,3)--(1,3,3)--(0,1,1)--cycle;
\filldraw[gray](2,0,2)--(2,2,2)--(1,3,3)--(1,0,3)--cycle;
\filldraw[gray](2,2,2)--(2,0,2)--(3,0,-2)--cycle;
\draw[thick](2,2,2)--(2,0,2);
\draw[thick](1,3,3)--(1,0,3);
\draw[thick](0,0,1)--(0,1,1);
\draw (1.5,-3) node{$\mathrm{(A)}$}; 
\end{tikzpicture}
\qquad
\begin{tikzpicture}[scale=1.1]
\draw[fill=gray](0,0)--++(0,1)--++(1,2)--++(1,-1)--++(1,-2)--cycle;
\draw[very thick](1,0)--++(0,3);
\draw[very thick](2,0)--++(0,2);
\draw[thick](0,0)--++(0,1);
\draw(0.5,0.5) node {$d(A_1)$};
\draw(1.5,0.5) node {$d(A_2)$};
\draw(2.5,0.5) node {$d(A_3)$};

\draw[densely dashed] (0,0) grid (3,3);
\draw (1.5,-0.35) node{$\mathrm{(B)}$}; 
\end{tikzpicture}
\qquad
\begin{tikzpicture}[scale=0.7]

\draw[thick](3,6.5)--++(8,0);
\draw(3,7.5)--++(2,-1)--++(0,-2);
\draw(8,4.5)--++(0,2)--++(2,2);
\draw(11,4.5)--++(0,2)--++(2,1); 
\draw (5,6.5) node {$\bullet$} ;
\draw (5,6.5) node[above] {$A_1$} ;
\draw (7,6.5) node {$\bullet$} ;
\draw (7,6.5) node[below] {$P$} ;
\draw (8,6.5) node {$\bullet$} ;
\draw (8,6.5) node[above] {$A_2$} ;
\draw (11,6.5) node {$\bullet$} ;
\draw (11,6.5) node[above] {$A_3$} ;

\draw (4,5.5) node {$1$};
\draw (2,7) node {$1+Y$};
\draw (6.5,7.5) node {$3+X+3Y$};
\draw (6.5,5.5) node {$3+X$};
\draw (9.5,5.5) node {$2+2X$};
\draw (12,5.5) node {$3X-2$};
\draw (11,7.5) node {$2+2X+2Y$};
\draw (6.5,3.25) node{$\mathrm{(C)}$}; 
\end{tikzpicture}
\qquad
\begin{center}
\caption {The extended Newton
  polyhedron $\widetilde{\A}$ of the  curve $C'$ (Example \ref{secondexample}) is drawn in $\mathrm{(A)}$. The projection of its faces gives us the subdivision of the
  Newton polygon of $C'$; see $\mathrm{(B)}$. The tropical curve $\Trop(C')$ is drawn
  in $\mathrm{(C)}$. The vertices $A_1,A_2,A_3$ have coordinates
  $(-2,0),(1,0),(4,0)$. The edge $A_1A_2$ has weight $3$, while
  the edge $A_2A_3$ has
  weight $2$. The point $P$ is $(0,0)=\Val((1,1))$. 
}
\label{example3}
\end{center}
\end{figure}
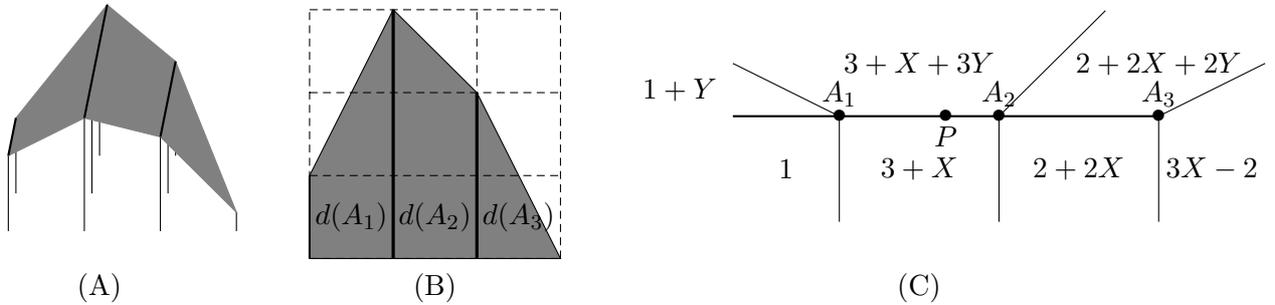

The curve $\Trop (C')$ is equal to the set of non-smooth points of the
function $$\Trop (F)=
\max(3+X+3Y,3+X+2Y,3+X+Y,3+X,2+2X+2Y,2+2X+Y,2+2X,1+Y,1,3X-2).$$ 
\end{ex}

The plane is divided by $\Trop(C')$ into
regions corresponding to the vertices of $\widetilde \A$. In Figure
\ref{example3}, the value of $\Trop (F)(X,Y)$ is written on each region. For example, $3X-2$ corresponds to the
vertex $(3,0,-2)$ of $\widetilde\A$.

A {\it tropical curve} $H\subset \TT^2$ is the non-smooth locus of a
function \eqref{eq_tropical} with finite $\A\subset\ZZ^2$. 
\begin{rem}
The tropical curves defined by the equations $\max (x,y,0)$ and $\max (2x,2y,0)$ coincide as sets, but the weights of the edges of the second curve are equal  to $2$, whereas for the first curve the weights of its edges are equal to $1$.
\end{rem}

Given a tropical curve $H$ as a subset of $\TT^2$ with weights on its edges (as we always assume in this paper), we can construct an equation, defining $H$. Then we construct the extended Newton polyhedron for $H$, using the same formula as for algebraic curves. The function defining $H$ is not unique, therefore the extended Newton polyhedron  for $H$ is defined up to a translation.

\begin{rem}\label{remark_passing}
When we pass from the set $\{(i,j,\val(a_{ij}))\}$ to $\widetilde \A$,
some information is lost. Nevertheless, we do not suppose that all
the points $\{(i,j,\val(a_{ij}))\}$ belong to the boundary of
$\widetilde \A$. 
\end{rem}

The reader should be familiar with the notions mentioned above,
or is kindly requested to refer to \cite{BIMS,oberwolfach,maclagan2015introduction}.

\subsection{Change of coordinates and $m$-fold points}
\begin{defi}
If the lowest term in the Taylor
expansion of $F$ at a point $p$ has degree $m$, then $m=\mu_p(C)$ is called the
{\it multiplicity} of $p$. The point $p$ is called an $m$-fold point
or a point of multiplicity $m$.
\end{defi}

Another way to say the same thing is to define $\mu_p(C)$
for $p=(p_1,p_2)$ as the maximal $m$ such that the polynomial $F$ belongs to the $m$-th
power of the ideal of the point $p$, i.e., $F\in\langle
x-p_1,y-p_2\rangle^m$ in the local ring of the point $p$.

\begin{ex}
The condition for a point $p$ to be of multiplicity one on $C$ means
that $p\in C$. Multiplicity greater than one implies that $p$ is a singular
point of $C$.
\end{ex}

\begin{ex}
Consider a curve $C'$ of degree $d$ given by an
equation  $$G(x,y)=\sum b_{ij}x^iy^j, 0\leq i,j,i+j\leq d.$$ The point
$(0,0)$ is of multiplicity at least $m$ on the
curve $C'$ if and only if $b_{ij}=0$ for all $i,j$ with $i+j<m$.  As a consequence,
for a given point $p\in (\KK^*)^2$, the condition that $\mu_p(C')\geq m$ can be
rewritten as a certain system of $\frac{m(m+1)}{2}$ linear equations in the
 coefficients $\{b_{ij}\}$ of $G$.
\end{ex}

\begin{ex}
Refer to Example \ref{secondexample}.
The point $p=(1,1)$ is a point of multiplicity $m=3$ on the curve $C'$.
This affects the subdivision of the Newton polygon of $C'$ in the following way:
\begin{itemize} 
\item The point $P=(0,0)$ belongs to an edge $E$ of the weight $m=3$.  
\item The sum of the areas of the faces dual to the vertices of
  $\Trop(C')$ on the extension of $E$ is $2+5/2+1=11/2$, which is greater than
$m^2/2=3^2/2$. 
\end{itemize}
\end{ex}
These two facts are particular incarnations of \hyperref[exertion_edges]{the Exertion Theorem for edges}.

\begin{lemma}
\label{toric_action}
Suppose $ad-bc=1$ where $a,b,c,d\in \ZZ$. The transformation $\Psi:(x,y)\mapsto (x^ay^b,x^cy^d)$ preserves multiplicity at the point $p=(1,1)$,
i.e., $\mu_{(1,1)}(C)=\mu_{(1,1)}(\Psi(C))$.
\end{lemma}

\begin{proof} Indeed, $\Psi$ defines an isomorphism in
  the local ring of $p=(1,1)$. One can prove this  by verifying that $\langle
  x-1,y-1\rangle=\langle x^ay^b-1,x^cy^d-1\rangle$ in the local ring
  of $p$.
\end{proof}

\begin{defi}
A map $f$ {\it tropicalizes} to a map $\Trop(f)$ if the
following diagram is commutative:
$$
\begin{CD}
\KK^2 @>f>> \KK^2\\
@VV{\Val}V @VV{\Val}V\\
\TT^2 @>{\Trop (f)}>> \TT^2
\end{CD}
$$

\end{defi}

\begin{prop}
\label{tropic_action}
A map $\Psi:(x,y)\mapsto (x^ay^b,x^cy^d)$ tropicalizes to the integer affine map
$\Trop(\Psi):(X,Y)\mapsto (aX+bY,cX+dY)$. \qed
\end{prop}
We define a new curve
$C'$ given by the equation $G(x,y)=0$, where
$G(x,y)=F(\Psi(x,y))$. Then the Newton polygon of $C'$ is the image of
$\Delta$ under 
$\bigl(\begin{smallmatrix}
a & c\\b & d
\end{smallmatrix}\bigr) \in SL(2,\ZZ)$, the same holds for the extended Newton
polyhedron, and $\Trop(C')=\Trop(\Psi)(\Trop(C))$. 
\begin{prop}
\label{translation1}
A map $\Psi_{r,q}: (x,y)\mapsto (rx,qy)$ with $r,q\in\KK^*$ tropicalizes to the map
$\Trop(\Psi_{r,q}):(X,Y)\mapsto(X+\val(r),Y+\val(q))$.\qed
\end{prop}

For $G(x,y)=\sum a'_{ij}x^iy^j$ defined as $G(x,y)=F(\Psi_{r,q}(x,y))$, an easy
computation gives $\val(a_{ij}')=\val(a_{ij}) + l(i,j)$ with
$l(i,j)=i\cdot \val(r)+j\cdot \val(q)$. This adds  the
linear function $l(i,j)$ to the extended Newton
polyhedron $\widetilde {\A}$,  therefore the subdivision of the Newton polygon
for $G$ coincides with the subdivision for $F$. This is not surprising
because of Proposition \ref{trop_prop} and the fact that $\Trop(\Psi_{r,q})$ is a
translation. Thus, $SL(2,\ZZ)$-invariant properties of the subdivision of $\Delta$ for
the curve $C$ with $\mu_{p}(C)=m$ for a
given point $p\in (\KK^*)^2$ do not depend on the point
$p$.  

\subsection{Lattice width and $m$-thick sets}

{\it Lattice width} is the most frequent notion in our arguments,
already proved to be a practical tool elsewhere. For example, the article
\cite{Castryck:2012rt} uses it to estimate the gonality of
a general curve with a given Newton polygon.
The minimal genera of surfaces dual to a given 1-dimensional
cohomology class in a three-manifold are related to the
lattice width of the Alexander polynomial of this class (\cite{friedl2008twisted,mcmullen2002alexander}). A good
survey of lattice geometry and related problems can be found in
\cite{barany}.

\begin{defi}
\label{def_dir} We denote by $\dirr$ the set of all directions in $\ZZ^2$. Each direction $u$ has a representative $(u_1,u_2)\in \ZZ^2$ with $\mathrm{gcd}(u_1,u_2)=1$. We will write $u\sim (u_1,u_2)$ in this case.
\end{defi}

Let us consider a compact set $B \subset \RR^2$.

\begin{defi}
\label{lattice_width}
{\it The lattice width} of  $B$ in a direction $u \in\dirr$ is
 $\omega_u(B)=\max\limits_{x,y\in B}(u_1,u_2)\cdot (x-y),$ where $u\sim (u_1,u_2).$
{\it The minimal lattice width} $\omega(B)$ is defined to be
$\min\limits_{u\in\dirr}\omega_u(B)$.
\end{defi}

Consider an interval $I$ with rational slope. Let $(p,q)\in\ZZ^2$ be a
primitive vector in the direction of $I$. The {\it  lattice length} of $I$ is its
Euclidean length divided by $\sqrt{p^2+q^2}$.

\begin{defi}
\label{mthick}
A set $B\subset \RR^2$ is called {\it $m$-thick} in the following cases:
\begin{itemize}
\item $B$ is empty,
\item $\conv(B)$ is 1-dimensional with rational slope and its lattice length is at least $m$,
\item $\conv(B)$ is 2-dimensional and for each $u\in \dirr$, if $\omega_u(\conv(B))=m-a_u$ with $a_u>0$, then $\conv(B)$ has two sides of
  lattice length at least $a_u$ and these sides are perpendicular to $u$.
\end{itemize}
\end{defi}

The relation between $m$-thickness and Euler derivatives is discussed in
Remark \ref{euler}.

\begin{prop}
\label{two_vertical}
If $B\subset \ZZ^2$ is $m$-thick and $\conv(B)$ is a polygon with at most
one vertical side, then $\omega_{(1,0)}(B)\geq m$. If $B$ is
$m$-thick and  $\conv(B)$ is a polygon without parallel sides, then $\omega(B)\geq m$. \qed
\end{prop}

\begin{lemma} \label{vertical_sides}
If $\mu_{(1,1)}(C)=m$ and $\omega_u(\A)=m-a$ for some $a>0, u\sim(u_1,u_2)$, then $C$
contains a rational component parametrized as $(s^{u_1},s^{u_2})$. 
\end{lemma}

\begin{proof} By Lemma~\ref{toric_action}, it is enough to prove
this lemma only for $u=(1,0)$. The degree of the polynomial $F(x,1)$ is $m-a$, but $F(x,1)$ has a root of
multiplicity $m$ at $1$, therefore $F$ is identically
zero on $y=1$, hence $F$ is divisible by $y-1$. 
Let $b$ be the maximal number such that $F$ is divisible by
$(y-1)^b$. Clearly $b\geq a$, otherwise we can repeat the above argument. Therefore $F$ is divisible by $(y-1)^a$, and both
vertical sides of $\conv(\A)$ have lattice length at least $a$. 
\end{proof}

\begin{cor}
\label{cor_mthick}
If $\mu_{(1,1)}(C)=m$, then the Newton polygon $\Delta$ of $C$ is $m$-thick.
\end{cor}

For a polynomial $G(x,y)=\sum b_{ij}x^iy^j$ we define its {\it support set} by
$\mathrm{supp}(G)=\{(i,j)|b_{ij}\ne 0\}.$

\begin{defi}
\label{amu}
For $\mu\in\RR$, denote by $\A_\mu$ the set $\{(i,j)\in\A|\val(a_{ij})\geq \mu\}$.
\end{defi}

The following lemma describes the set of valuations of the
coefficients $a_{ij}$ of $F(x,y)$. 
\begin{lemma}[$m$-thickness lemma]
\label{thickness_lemma}
\addcontentsline{toc}{subsection}{M-thickness Lemma}
If $\mu_{(1,1)}(C)=m$, then for each real number $\mu$ the set $\A_\mu$
is $m$-thick (Def.~\ref{mthick}).
\end{lemma}

\begin{proof} We will find a polynomial $G$ with
$\mathrm{supp}(G)=\A_\mu$, which defines a curve passing through
$(1,1)$ with multiplicity $m$. Then Corollary \ref{cor_mthick} concludes the proof.  Let us consider the set of linear
conditions in the coefficients $a_{ij}$ imposed by the fact that
$\mu_{(1,1)}(C)=m$. If there is no required polynomial $G$, then by setting all
the coefficients $a_{ij}$ to $0$ for $(i,j)\in\A\setminus\A_\mu$, we see
that the above system of linear equations would imply that $a_{i'j'}=0$ for some
$(i',j')\in\A_\mu$. That would mean that there exists an equation $\sum
\lambda_{ij}a_{ij}=a_{i'j'},\lambda_{ij}\in\QQ,(i,j)\in \A\setminus\A_\mu$ which is a
consequence of the above system. The latter
leads us to the contradiction, because for the polynomial $F$ we have
$\val(\lambda_{ij}a_{ij})<\mu\leq\val(a_{i'j'})$ for $(i,j)\in
\A\setminus\A_\mu$ (see Remark \ref{rem_max}). The attentive reader
can notice that $\A_\mu$ is a {\it flat} in the matroid
corresponding to the above linear conditions. Indeed, no dependent set
intersects $\A_\mu$ in exactly one element, because the valuation of
this element would be strictly bigger than the valuations of the other
elements in this dependent set. 
\end{proof}

\subsection{A lemma about concave functions}

Suppose that $h:[a,b]\to\RR$ is a concave and piecewise smooth function on the interval
$[a,b]$. Define $\hat{h}_{[a,b]}(x)$ as the length of the subinterval
of $[a,b]$ where
the values of $h$ are at least $h(x)$, i.e., $\hat{h}_{[a,b]}(x)= \mathrm{measure}\{y\in [a,b]| h(y)\geq h(x)\}$. 
\begin{lemma}
\label{function}
Suppose that $h$ attains its maximal value at a unique point.  Then
$\int_a^b\hat{h}_{[a,b]}(x)dx = (b-a)^2/2$.
\end{lemma}
\begin{proof} Without loss of generality $h(a)\geq h(b)=0$. Let $q$
be the point where the maximum of $h$ is attained. On the intervals $[a,q]$ and $[q,b]$
the function $h$ is invertible. Call the respective
inverses $f_1,f_2$, that is $f_1(h(x))=x$ for $x\in [a,q]$ and
$f_2(h(x))=x$ for $x\in [q,b]$. For $y\in[0,h(a)]$, we define
$f_1(y)=a$. Hence
$f_1(h(q))=f_2(h(q))=q, f_1(0)=a,f_2(0)=b$.
Let $H(y)=f_2(y)-f_1(y)$; note that $H(y)=\hat h(f_1(y))=\hat h(f_2(y))$.
Finally, we integrate $\hat{h}_{[a,b]}(x)$ along the $y$-axis. In
between, we change
the measure in the integration. The integral becomes
\[\int_a^b\hat{h}_{[a,b]}(x)dx=\int_{h(q)}^{0}
(h_2(y)-h_1(y))d(h_2(y)-h_1(y))
=\int^0_{h(q)} H(y)d(H(y)) =
\frac{H^2(0)}{2} = \frac{(b-a)^2}{2}.\qedhere\]
\end{proof}

\begin{cor}\label{cor_integral}
If $h(a')=h(b')$ for some $a'<b'$ in $[a,b]$, then $$\int_a^{a'}\hat{h}_{[a,b]}(x)dx+\int_{b'}^{b}\hat{h}_{[a,b]}(x)dx = \frac{1}{2}((b-a)^2-(b'-a')^2).$$
\end{cor}

\begin{proof}
We proceed as in the proof of the lemma, and \[\int_a^{a'}\hat{h}_{[a,b]}(x)dx+\int_{b'}^{b}\hat{h}_{[a,b]}(x)dx=\int^{0}_{h(a')} H(y)d(H(y))=\frac{1}{2}((b-a)^2-(b'-a')^2).\qedhere \]
\end{proof}

\begin{prop}
\label{concavity}
If $h$ is linear with non-zero slope on an interval
$[a',b']\subset [a,b]$, then the function $\hat{h}$ is concave on $[a',b']$.
\end{prop}

\begin{proof} Without loss of generality we suppose that
$a'<b', f(a')<f(b')$. It is enough to check that
$\hat{h}\big(\tfrac{1}{2}(x+y)\big)\geq\tfrac{1}{2}\big(\hat{h}(x)+\hat{h}(y)\big)$ for
$x,y\in [a',b']$.  Since $h\big(\tfrac{1}{2}(x+y)\big)=\tfrac{1}{2}(h(x)+h(y))$, we have

\begin{align*} \hat{h}\Big(\tfrac{1}{2}(x+y)\Big)&=f_2\bigg(h\Big(\tfrac{1}{2}(x+y)\Big)\bigg)-f_1\bigg(h\Big(\tfrac{1}{2}(x+y)\Big)\bigg)=f_2\bigg(h\Big(\tfrac{1}{2}(x+y)\Big)\bigg)-\tfrac{1}{2}\Big(f_1\big(h(x))+f_1(h(y)\big)\Big)\\
&\geq \tfrac{1}{2}\Big(f_2(h(x))+f_2(h(y))\Big)-\tfrac{1}{2}\Big(f_1(h(x))+f_1(h(y))\Big)=\tfrac{1}{2}\Big(\hat{h}(x)+\hat{h}(y)\Big),\end{align*}
because $f_2\circ h$ is
concave on the interval $[x,y]$. Note that linearity of $h$ and $f_1$
on $[x,y]$ is crucial, since $f_1$ has a negative
coefficient.  
\end{proof}

\section{Formulation of main theorems}
\label{formulation}
In this section, we state the main results of this paper. For the terminology of faces, vertices, edges, and the duality among them, refer to Proposition \ref{trop_prop}. 

\subsection{Influenced sets}
We consider a tropical curve $H\subset \TT^2$ and a point $Q\in H$.
\begin{defi}\label{defi_lq}Let $l_Q(u)$ be the line through $Q$ in the direction $u\in\dirr$. 
Take the connected component, containing $Q$, of the intersection
  $H\cap l_Q(u)$. We call this component {\it the long edge
    through $Q$ in the direction $u$} and denote it by $E_Q(u)$. 
\end{defi}

\begin{defi}
\label{def_I}
For each $u\in\dirr$ we denote by $\I_Q(u)$ the set of vertices of $H$ which belong
to the long edge $E_Q(u)$. Define $\I(Q)=\bigcup_{u\in\dirr}\I_Q(u)$.
\end{defi}

Note that $\I(Q)$ is not a multiset; it contains only one copy of
$Q$.  Examples of $\I(P)$ are presented in Figure \ref{governorship}. 
On the left we see one long edge $E_P((1,0))$ and
$\I(P)$ consists of 7 vertices, and above we see 7 corresponding
faces in the subdivision of $\Delta$. On the right,  we see long edges
$E_P((1,0)),E_P((0,1)),E_P((-1,1))$. Each of the long edges 
$E_P((1,2))$ and $E_P((-3,-2))$ consists of only one edge. In Example \ref{secondexample}, $E_P((1,0))$ is the union of
the horizontal edges of $\Trop(C')$ and $\I(P)$ is the set of all
vertices of $\Trop(C')$.

\begin{defi}
For a point $Q\in H$ we define $\Inf(Q)=\bigcup\limits_{V\in \I(Q)}d(V),$ the union of the faces of the Newton polygon of $H$, dual to the vertices in $\I(Q)$.
\end{defi}

\begin{defi}
\label{def_influence_edge}
For a point $Q
\in H$ which is not a vertex of $H$,
we define $$\area(\Inf(Q))=\sum\limits_{F\in\Inf(Q)}\area(F).$$
\end{defi}

Note that $\area(\Inf(Q))$ depends only on $H$ and does not depend on a
particular choice of an equation defining $H$. Also, if $Q$ belongs to an edge $E$ of $H$ and $Q$ is not a vertex of $H$,
then $\I_u(Q)=\I(Q)$ where $u$ is the direction of $E$. Indeed, for any other direction $v$ not collinear to $u$, 
the connected component of $Q$ in the intersection $H\cap
l_P(v)$ is just $Q$.

Recall that if $Q$ is a vertex of  $H$, then $d(Q)$ is a face dual to $Q$ in the
subdivision of $\Delta$.

\begin{defi}
\label{def_influence_vertex}
If $Q$ is a vertex of $H$, we define \begin{align*}\area(\Inf(Q))&=\sum\limits_{F\in \Inf(Q)}\area(F)+\area(d(Q)),\\
\area^{*}(\Inf(Q))&=\sum\limits_{F\in \Inf(Q)}\area(F).\end{align*}  
\end{defi}

From the point of view of combinatorics,
studying $\area^{*}(\Inf(Q))$ is more natural, whereas $\area(\Inf(Q))$ is
motivated by Nagata's conjecture (see \cite{2013arXiv1310.6684K} for details). The name $\Inf(P)$ is chosen because the linear constraints, imposed by
the fact $\mu_p(C)=m$, asymptotically {\it influence} (c.f. Remark
\ref{motivation}) the coefficients $a_{ij}$ where
$(i,j)\in\Inf(P), P=\Val(p)$.

\subsection{Multiplicity of a tropical point in the intermediate
  sense}
 
Consider a tropical curve $H$ given by a tropical polynomial
$\Trop(F)$. Using $\Trop(F)$, we construct the
extended Newton polyhedron $\widetilde\A$ for $H$. 

\begin{defi}
\label{def_wamu}
We denote by $\widetilde{\A}_\mu$  the $xy$-projection of the section of
$\widetilde\A$ by the plane $z=\mu$. 
\end{defi}
Note that $\A_\mu$ (Def.~\ref{amu}) is contained in $\widetilde\A_\mu$. In Figure \ref{extNewton} (below), the set $\widetilde\A_\mu$ is colored in gray.

\begin{defi}
\label{def_intermediate}
A point $P=(0,0)$ on the tropical curve $H$ is of
multiplicity at least $m$ {\it in the intermediate sense} (we write
$\mu^{\trop}_P(H)\geq m$) if for each $\mu\in
\RR$ the set $\widetilde\A_\mu$ is $m$-thick (Def.~\ref{mthick}).
\end{defi}
Using Proposition
\ref{translation1}, we can use this definition for any other point of $P\in\Val((\KK^*)^2)$, after an appropriate change of coordinates.
\begin{lemma}
\label{lemma_implication}
If $\mu_p(C)=m$ and $P=\Val(p)$, then $\mu^{\trop}_P(\Trop(C))\geq m$.
\end{lemma}

\begin{cor}
If a point $P$ on a tropical curve $H\subset \TT^2$ is of multiplicity at least $m$ in the $\KK$-extrinsic
sense (Def.~\ref{def_extrinsic}), then $P$ is of multiplicity at least $m$ for $H$ in the intermediate
sense.
\end{cor}
Unfortunately, this lemma does not immediately follow from Lemma \ref{thickness_lemma}. 

\subsection{Exertion Theorems}

If $\omega(\A)<m$, i.e., $\omega_u(\A)<m$ for some
$u\sim(u_1,u_2)$ (Def.~\ref{def_dir}), and $\mu_p(C)=m, p=(p_1,p_2)$,
then Lemma \ref{vertical_sides} asserts
that $C$ contains a rational component with parameterization
$(p_1s^{u_1},p_2s^{u_2})$. We also prohibit such cases on the tropical side of the story.
\begin{defi}
\label{admissible}
A tropical curve is {\it admissible} if the minimal lattice width
(Def.~\ref{lattice_width}) of its Newton polygon is at least $m$.
\end{defi}

The following theorems estimate the total area of the region of influence
 of $P$ in $\Delta$. The point $P$ {\it exerts} its influence on the faces whose area
is counted in the theorem, whence the name.

\begin{theo} [Exertion Theorem for edges]
\label{exertion_edges}
\addcontentsline{toc}{subsection}{Exertion Theorem for edges}
If $H$ is admissible, $\mu^{\trop}_P(H)=m$ (Def.~\ref{def_intermediate}), and
$P$ is not a vertex of $H$, then $\area(\Inf(P))\geq
\frac{1}{2}m^2$ (Def.~\ref{def_influence_edge}). Furthermore, if $P$ belongs to an edge $E \subset H$, then the lattice length of $d(E)$ is at least $m$.
\end{theo} 

In this case we see a collection of faces with parallel sides in the
subdivision of $\Delta$; see Figure \ref{governorship}(A).

\begin{theo} [Exertion Theorem for vertices]
\label{exertion_vertices}
\addcontentsline{toc}{subsection}{Exertion Theorem for vertices}
If $H$ is admissible, $\mu^{\trop}_P(H)=m$, and the point $P$ is a vertex of $H$, then $\area^{*}(\Inf)(P)\geq
\frac{3}{8}m^2$ and $\area(\Inf(P))\geq \frac{1}{2}m^2$ (Def.~\ref{def_influence_vertex}).
\end{theo}

Here we will see a collection of faces like in Figure
\ref{governorship}(B). The Exertion theorems are valid only for admissible curves. The following example
illustrates this problem. 
\begin{ex}
\label{counterexample}
Consider a curve $C'$ defined by the polynomial 
$F_k(x,y)=(x-1)^k(y-1)^{m-k}$. Clearly, $\mu_{(1,1)}(C') = m$ but the curve
$\Trop(C')$ is not admissible. The Newton polygon of $F_k$ is the
rectangle with vertices $(0,0),(k,0),(0,m-k),(k,m-k)$, it is
$m$-thick and its area is $k(m-k)$ which is always less than
$\frac{3}{8}m^2$. The curve $C'$ consists of the line $x=1$ with
multiplicity $k$ and the line $y=1$ with multiplicity $m-k$. The
tropical curve
$\Trop(C')$ consists of the vertical line of weight $k$ and the
horizontal line of weight $m-k$. Note that Lemma~\ref{lemma_implication} holds in this example.
\end{ex}

\section{Intrinsic definition of a tropical $m$-fold point}

The {\it multiplicity} $m(P)$ of the point $P$ of the intersection of two lines in directions $u,v\in\dirr$ is  $|u_1v_2-u_2v_1|$ where $u\sim (u_1,u_2),v\sim (v_1,v_2)$ (Def.~\ref{def_dir}).

Given two tropical curves $A,B\subset \TT^2$ we define their {\it stable intersection} as
follows. Let us choose a generic vector $v$. Then we consider the curves $T_{tv}A$
 where $t\in\RR, t\to 0$ and $T_{tv}$ is translation by
the vector $tv$. For a generic small positive $t$, the intersection $T_{tv}A\cap B$ is
transversal and consists of points $P_i^t, i =1,\dots,k$ with multiplicities $m(P_i^t)$. 

\begin{defi}[c.f. \cite{sturm}]
For each connected component $X$ of $A\cap B$, we define {\it the
  local stable intersection of $A$ and $B$ along $X$} as
$A\cdot_XB=\sum_i m(P_i^t)$ for $t$ close to zero, where the sum
runs over $\{i|
\lim_{t\to 0} P_i^t\in X\}$. For a point $Q\in A$, we define $A\cdot_QB$ as
$A\cdot_XB$, where $X$ is the connected component of $Q$ in the
intersection $A\cap B$.
\end{defi}

\begin{defi}
A generalized tropical line is the non-smooth locus of a function
\eqref{eq_tropical} with $\A\subset \ZZ^2$ such that $\A$ is an
interval of lattice length $1$ or  $|\A|= 3,\area(\conv(\A))=\frac{1}{2}$.
\end{defi}

\begin{prop}
\label{prop_intersection}
Let $Q$ be a vertex of a tropical curve $H$. If the face $d(Q)$ has no vertical sides, and $L$ is the usual horizontal line through $Q$, then $H\cdot_QL=\omega_{(1,0)}(d(Q))$.
\end{prop}
\begin{proof}
This follows from a direct computation and Proposition \ref{trop_prop}.
\end{proof}

\begin{defi}
\label{def_intrinsic}
A point $P$ on a tropical curve $H$ is of
multiplicity at least $m$ {\it in the intrinsic sense} if for each generalized tropical line $L$ through $P$ we have $L\cdot_PH\geq m$.
\end{defi}

\begin{rem}
Given $Q\in H$, we call {\it the
  tangent cone $TC(Q)$ at $Q$} the connected component of $Q$ in the intersection $H\cap\bigcup_{u\in\dirr}\{l_Q(u)\}$ (Def.~\ref{defi_lq}). Note that only the vertices of $H$ in $TC(Q)$ contribute to the multiplicity of $Q$ in the intrinsic sense. Also, this set of vertices coincides with $\I(Q)$ (Def.~\ref{def_I}).
\end{rem}

\begin{prop}
Let $P$ be of multiplicity $m$ in the intrinsic sense. If $P$ is a
vertex of $H$, then $d(P)$ is $m$-thick (Def.~\ref{mthick}). If $P$ is
not a vertex of $H$, then the edge of $H$ containing $P$ is of weight
at least $m$. 
\end{prop}
\begin{proof}
For each $u\in\dirr$, we can find  a generalized tropical line $L$ such that $P$ is the
vertex of $L$, and $L$ has an edge in the direction
$u$. Like in Proposition \ref{prop_intersection}, a direct calculation of $L\cdot_PH$ finishes the proof.
\end{proof}

Now consider Example \ref{example3}. The edge with $P$ has weight 3,
therefore the stable intersection with each non-horizontal line is at
least 3. The stable intersection of $H$ with the horizontal line through $P$ is exactly
the width of the Newton polygon of the curve in the direction $(1,0)$. 

Consider an edge of $H$ through $P$. Without loss of generality
we can suppose that this edge is horizontal. Let $A_1$ (resp. $A_2$) be the
leftmost (resp. rightmost)
vertex of $H$ on the horizontal long edge $E_P((1,0))$ (Def.~\ref{defi_lq}).

\begin{prop} [c.f. Lemma \ref{boundary}]
\label{prop_width}
If $P$ is of multiplicity $m$ in the intrinsic sense and $E_P((1,0))=A_1A_2$, then the difference between $x$-coordinates of the leftmost vertex of
$d(A_1)$ and rightmost vertex of $d(A_2)$ is at least $m$.
\end{prop}
\begin{proof} Let $L$ be the usual line containing $E$. A direct calculation of $L\cdot_PH$ concludes the proof.
\end{proof}

\begin{prop} Suppose that $P\in H$ is not a vertex of $H$. Let $P$ belongs to an edge $E$ of $H$ with endpoints $A_1$ and $A_2$.
Let  $P$  be of
multiplicity $m$ for $H$ in the intrinsic sense. Suppose that $E_P((1,0))=E$. Then
$\area(d(A_1))+\area(d(A_2))\geq \frac{1}{2}m^2$.
\end{prop}
\begin{proof}
The lattice length of $d(E)$ is at least $m$ and
the sum of the heights of $d(A_1)$ and $d(A_2)$ is at least $m$ by Proposition \ref{prop_width}. Therefore $$\area(d(A_1))+\area(d(A_2))\geq m\cdot m/2.$$
\end{proof}

\section{Two combinatorial lemmata}

\begin{defi}
\label{def_def}
The {\it defect} of $B\subset \ZZ^2$ in a direction $u\in \dirr$ is $\deff_u(B)=\max(m-\omega_u(B),0)$.
\end{defi}

This section is devoted to the proofs of the following statements.

\begin{lemma} 
\label{area4}
For an $m$-thick (Def.~\ref{mthick}) lattice polygon $B$  we have
$$\area(\conv(B))+\frac{1}{2}\sum_{u\in\dirr}\deff_u(B)^2\geq
\frac{3}{8}m^2.$$
\end{lemma}

\begin{lemma}
\label{area3}
For an $m$-thick lattice polygon $B$ we have
$$2\cdot\area(\conv(B))+\frac{1}{2}\sum_{u\in\dirr}\deff_u(B)^2\geq
\frac{1}{2}m^2.$$
\end{lemma}

Unfortunately, though the proofs use only standard combinatorial
arguments, they are cumbersome and rather tedious. Thus the reader is
recommended to skip this section while reading this paper the first time.

\subsection{Using the direction $(0,1)$ or the direction $(1,1)$} 
Suppose that $B$ is not $(m+1)$-thick and the minimal
lattice width $a\leq m$
of $B$ is attained in the horizontal direction. Using the $m$-thickness property, we can find two points $M,L$ on the
left vertical side of $B$ and two  points
$N,K$ on the right vertical side in such a way (Figure \ref{est1}(A)) that the distances $ML$ and $NK$
are equal to $m-a$, so $MNKL$ is a parallelogram. Let us call it the {\it
  initial} parallelogram. Note that in the case $a=m$ we have a
degenerate initial parallelogram with $M=L,N=K$.

\begin{defi}
Denote by $x(A)$ (resp.~$y(A))$ the $x$-coordinate  (resp.~$y$-coordinate) of a point $A\in \ZZ^2$.
\end{defi}

Let $b=y(M)-y(N)$. Applying a suitable coordinate change in $SL(2,\ZZ)$ we may assume that $0\leq b<a$; see Figure \ref{est1}(A).

\begin{prop}\label{prop_paral}
The width $\omega_{(0,1)}(MNKL)$ of the initial parallelogram $MNKL$ in the
direction $(0,1)$ is equal to
$m-a+b$. The width $\omega_{(1,1)}(MNKL)$ is equal to $m-b$.\qed
\end{prop}

\begin{figure} [htbp]
\begin{center}
\begin{tikzpicture}[scale=0.4]
\draw(2,2)--(2,7)--(6,5)--(6,0)--cycle;
\draw(2,2) node[below]{$L$};
\draw(2,7) node[above]{$M$};
\draw(6,5) node[above]{$N$};
\draw(6,0) node[below]{$K$};
\draw[<->](2,3.5)--(6,3.5);
\draw(4,4) node{$a$};
\draw(2,5.5) node{$m-a$};
\draw (5.5,2.5) node{$m-a$};
\draw[<->](7.5,7)--(7.5,5);
\draw(8,6)node{$b$};
\end{tikzpicture}
\qquad
\begin{tikzpicture}[scale=0.3]
\draw(3,4)--(3,9)--(4,10)--(6,10)--(9,7)--(9,2)--cycle;
\draw(3,4)--(7,1)--(9,1)--(9,2);
\draw(3,9)--(9,7);
\draw[<->](3,5.5)--(9,5.5);
\draw (6,6) node{$a$};
\draw(4,10)node[above left]{$M_1$};
\draw(6,10)node[above right]{$M_2$};
\draw(4,10)node[above left]{$M_1$};
\draw(9,1)node[below right]{$K_1$};
\draw(7,1)node[below left]{$K_2$};
\draw[<->](4,9.7)--(6,9.7);
\draw[<->](7,1.3)--(9,1.3);
\draw (5,9.3) node{$x$};
\draw (8,1.7) node{$x$};
\draw(1,9.5)node{$x_1$};
\draw[<->](1.8,10)--(1.8,9);
\draw(11,1.5)node{$x_2$};
\draw[<->](10.2,1)--(10.2,2);
\end{tikzpicture}
\qquad
\caption {The initial
parallelogram $MNKL$ is depicted on the left. The set $B$ is $m$-thick. Therefore,
by taking into consideration
$\omega_{(0,1)}(B)$, we find a polygon $MM_1M_2NKK_1K_2L$, which is a subset of $B$.}
\label{est1}
\end{center}
\end{figure}
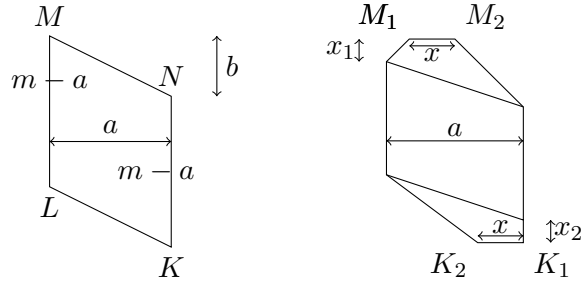

Suppose that $\omega_{(0,1)}(B)=m-x$. Thus, by Proposition \ref{prop_paral}, $x\leq a-b$,
and $B$ must have two horizontal sides $M_1M_2,K_1K_2$, whose lengths are at
least $x$.  Note that it is possible that $x=0$; in that case
we can choose $M_1=M_2\in B,K_1=K_2\in B, y(M_1)-y(K_1)=m$. So, $B$
contains a polygon $MM_1M_2NKK_1K_2L$. A particular example of such a polygon is shown in Figure \ref{est1}, right side. Let
$x_1=x(M_1)-x(M),x_2=x(K)-x(K_1)$. The inequality $x_1+x_2\geq m- (m-a+b+x)$ holds because $B$ is $m$-thick. All the notation is presented in Figure \ref{est1} and this picture serves as the main illustration tool for the following computations.

Note that 
\begin{equation}
\label{eq_s01}
\area(MM_1M_2NKK_1K_2L\setminus MNKL) \geq
a(x_1+x_2)/2+x(b+x_1+b+x_2)/2,
\end{equation} 
and the minimum is attained if the bottom
horizontal edge is in the extremal right position (like at the bottom
in Figure \ref{est1}(B)), and the top edge is in the extremal left
position. Look at the top of Figure \ref{est1}(B)): we minimize
the area of $MM_1M_2NKK_1K_2L$, preserving $MNKL$ and $x_1,x_2$. For that, we should move the
interval $M_1M_2$ to the left as much as possible, while preserving the
convexity of $MM_1M_2NKK_1K_2L$.

\begin{defi}
Define $S_{(0,1)}=\frac{1}{2}\deff_{(0,1)}(B)^2 + 2\cdot\area(B\setminus(MNKL))$.
\end{defi}

Using \eqref{eq_s01}, we see that  
\begin{equation}
S_{(0,1)}\geq x^2/2+ a(x_1+x_2)+x(b+x_1+b+x_2)
\geq a(a-b)+xb-x^2/2.
\end{equation}

\begin{rem}
\label{rem_quadr}
If $c_2<0$, then a function $f(x)=c_2x^2+c_1x+c_0$ defined on an interval
$[c_3,c_4]$ always attains its minimum at $c_3$ or $c_4$.
\end{rem}

We will extensively use this fact below. In particular, $x\in [0,a-b]$ and we have $$S_{(0,1)} \geq\min(a(a-b),
a(a-b) +(a-b)(b-\frac{a-b}{2}).$$ Moreover, if $b\geq a/3$, then $S_{(0,1)} \geq a(a-b)$.  If
$b\leq a/3$, then $$S_{(0,1)}\geq a(a-b)+(3b-a)(a-b)/2.$$

\begin{lemma}\label{lemma_first}
If $b\leq a/3$, then $S_{(0,1)}\geq a^2/2$.
\end{lemma}
\begin{proof}
In this case $S_{(0,1)}\geq a(a-b)+(3b-a)(a-b)/2$.
It follows from Remark \ref{rem_quadr} that it is enough to consider the cases $b=0$ and $b=a/3$. 
\end{proof}

We repeat the above procedure for the direction $(1,1)$. We define $y
= m-\omega_{(1,1)}(B)$. Then, let
$N_1N_2,L_1L_2$ be the vertices of  two sides  of
$B$, perpendicular to the direction $(1,1)$ and $y_1,y_2$ be the
increments of $\omega_{(1,1)}$ obtained by adding $N_1,N_2,L_1,L_2$ to
$MNKL$. Then, $y_1+y_2\geq b-y$ because $B$ is $m$-thick. On Figure \ref{est3} we have $y_1=0,y=1$; note that $y_2=2$
because  $\omega_{(1,1)}(\{(0,0),(1,1)\})=2$.

\begin{defi}
We denote $S_{(1,1)}=\frac{1}{2}\deff_{(1,1)}(B)^2 +
2\cdot\area(B \setminus(MNKL)).$
\end{defi} 

By direct calculation of the areas of the triangles $L_1L_2K, LL_1K, MN_1N_2,MN_2N$, we obtain  \begin{equation}
\label{eq_S11}
S_{(1,1)}\geq y^2/2+a(y_1+y_2)+ y(a-b+y_1+a-b+y_2) \geq
-y^2/2+ab+y(a-b).
\end{equation}

\begin{prop} The following inequalities hold:
1) if $b\leq 2a/3$, then $S_{(1,1)}\geq ab$,

2) if $b\geq 2a/3$, then $S_{(1,1)}\geq ab+b(2a-3b)/2$. 
\end{prop}
\begin{proof} It follows from \eqref{eq_S11}, Remark \ref{rem_quadr}, and the fact that $0\leq y\leq b$.
\end{proof}

\begin{lemma}\label{lemma_second}
If $b\geq 2a/3$, then $S_{(1,1)}\geq a^2/2.$ 
\end{lemma}
\begin{proof}
Again, if $b=a$, then we obtain $S_{(1,1)}\geq a^2/2$; for $b=2a/3$, we get $S_{(1,1)}\geq 2a^2/3$.
\end{proof}

\begin{lemma}
\label{area2}
The following inequality holds: $$2\cdot\area(B\setminus(MNKL))
 +\frac{1}{2}\sum_{\substack{u\in\dirr,\\ u\ne (1,0)}}  \deff_u (B)^2 \geq \frac{a^2}{2}.$$
\end{lemma}

\begin{proof} Indeed, if $a/3\leq b\leq 2a/3$, then $S_{(0,1)}+S_{(1,1)}\geq a^2$ and we are done. Two other cases are covered by Lemmata \ref{lemma_first}, \ref{lemma_second}.
\end{proof}

\begin{proof}[Proof of Lemma \ref{area3}] It follows from the previous
  lemma that 
\begin{align*}
2\cdot\area(\conv(B))&+\frac{1}{2}\sum_{u\in\dirr}\deff_u(B)^2\geq 2\cdot\area(MNKL)+\frac{1}{2}a^2+\frac{1}{2}\deff_{(1,0)}(B)^2 \\ &\geq 2(a(m-a))+\frac{1}{2}a^2+\frac{1}{2}(m-a)^2 \geq \frac{1}{2}m^2 + a(m-a),\end{align*}
and $a(m-a)\geq 0$.\end{proof}

\subsection{Using both directions $(0,1)$ and $(1,1)$}
Now we will use the widths of $B$ in the directions
$(0,1),(1,1),(1,0)$ at the same time. Consider the directions
$(0,1),(1,1)$, and define $x,y,x_1,y_1,x_2,y_2$ as in the previous subsection. Now, $B$ contains
the polygon $s(B)=MM_1M_2NN_1N_2KK_1K_2LL_1L_2$. Some of its vertices are allowed to
coincide. Refer to Figure \ref{est3}.  We assume that the
polygon $s(B)$ satisfies the condition of $m$-thickness in the
directions $(0,1),(1,0),(1,1)$. Our goal is to find an estimate for
the area of $s(B)$ in terms of $m,a,x,y$. We can suppose that $s(B)$, with the above requirements,
is the minimal polygon by area.

\begin{figure} [htbp]
\begin{center}
\begin{tikzpicture}[scale=0.6]
\draw(2,5)--(2,9)--(9,6)--(9,2)--cycle;
\draw[<->](2,6.5)--(9,6.5);
\draw(5,6) node{$a$};
\draw(2,9)--(2,10)--(4,10)--(8,7)--(9,6);
\draw(2,5)--(3,2)--(4,1)--(6,1)--(9,2);

\draw[<->](9.6,6)--(9.6,9);
\draw (10,7.5) node{$b$}; 
\draw[<->](9.5,1)--(9.5,2);
\draw(10,1.5) node{$x_2$};

\draw[<->](0.5,4.5)--(1.5,5.5);
\draw[<->](1.5,9)--(1.5,10);
\draw(1,9.5) node{$x_1$};
\draw(2,10.5) node{$M_1$};
\draw(4,10.5) node{$M_2$};

\draw(0.5,5.5) node{$y_2$};
\draw(3,2) node[above right]{$L_2$};
\draw(4,1) node[above right]{$L_1$};
\draw(6,1) node[below right]{$K_1$};

\draw[<->](2.5,1.5)--(3.5,0.5);
\draw(2.5,0.5) node{$y$};
\draw[<->](8.5,7.5)--(9.5,6.5);
\draw(9.3,7.5) node{$y$};

\draw[<->](4,0.7)--(6,0.7);
\draw(5,0.2) node{$x$};
\draw[<->](2,9.7)--(4,9.7);
\draw(3,9.3) node{$x$};

\draw(8,7) node[below left]{$N_1$};
\draw(9,6) node[below left]{$N_2=N$};
\draw(3,2) node{$\bullet$};
\draw(4,1) node{$\bullet$};
\draw(8,7) node{$\bullet$};
\draw(9,6) node{$\bullet$};
\draw[ultra thick](2,10)--(4,10);
\draw[ultra thick](8,7)--(9,6);
\draw[ultra thick](3,2)--(4,1)--(6,1);
\draw [densely dotted](0,0) grid (10,11);
\end{tikzpicture}
\qquad
\caption {In this example $m=11,a=7$, and $y_1=0$ (therefore $N=N_2$). The vertices $MNKL$ are as in Figure \ref{est1}(A), the
vertices $M_1M_2,K_1,K_2$ are as in Figure \ref{est1}(B), and $L_1$ an $K_2$ coincide. We are looking for the minimum of the sum of the area of this
polygon and $\frac{1}{2}(x^2+y^2)$. }
\label{est3}
\end{center}
\end{figure}
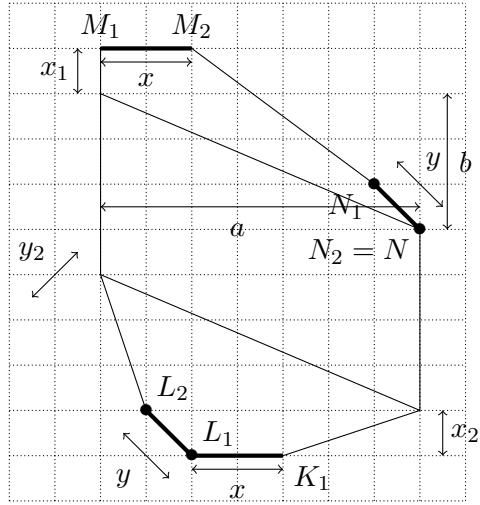
\begin{lemma}
The pairs of intervals  $M_1M_2,N_1N_2$ and $K_1K_2,L_1L_2$ either
share a common vertex (like $K_1K_2$ and $L_1L_2$ in the bottom of Figure
\ref{est3}), or are
maximally far from each other (like $M_1M_2$ and $N_1N_2$ at the top of the
picture).
\end{lemma}

\begin{proof} This lemma follows from the fact that the area
  changes linearly when we move the sides $K_1K_2,L_1L_2,M_1,M_2,N_1,N_2$, preserving  the distances
  $x,y, x_1,y_1,x_2,y_2$. \end{proof}

Let $A_1$ denote the minimal area of the top augmented
piece ($MM_1M_2N_1N_2N$) when $N_1N_2$ and $M_1M_2$ are maximally far from
each other (Figure \ref{est3}, top). Let $A_2$ denote the minimal area of the bottom augmented
piece ($LKK_1K_2L_1L_2$) when $L_1L_2$ and $K_1K_2$ are maximally far from
each other. Let $A_3$ denote the minimal area of the top augmented
piece when $N_1=M_2$. Let $A_4$ denote the minimal area of the bottom augmented
piece when $L_1=K_2$ (Figure \ref{est3}, bottom).  

\begin{lemma}
For $A_1,A_2,A_3,A_4$ defined above, we have $A_1-A_3=A_2-A_4$.
\end{lemma}

\begin{proof} Computing $\omega_{0,1}(B),\omega_{1,1}(B)$, we get
relations $x_1+x_2=a-b-x, y_1+y_2=b-y$.
Now, by direct computations we obtain 
$$A_1 
=\frac{1}{2}(ax_1-yx_1-yb+yy_1+ay_1+ay+xx_1+xb-xy_1-xy).$$

Replacing $x_1$ with $x_2$ and
$y_1$ with $y_2$ and using the above relations we obtain $A_2$: 
$$A_2 =\frac{1}{2}(a^2-ax_1+yx_1+by-yy_1-y^2-ay_1-ay-xx_1-x^2-xb+xy_1+xy).$$

For $A_3,A_4$ we get
$$A_3=\frac{1}{2}(yy_1+xx_1+ax_1+ab-b^2-bx_1+by_1),$$

$$A_4=\frac{1}{2}(-yy_1-y^2-xx_1-x^2+a^2-ab-ax_1+bx_1-by_1+b^2).$$

It is straightforward to see that $A_1-A_3=A_4-A_2$.\end{proof}

If $A_1<A_3$, then $A_4<A_2$. Therefore, the minimal total sum of the areas of the augmented
pieces is  $A_1+A_4$ or $A_2+A_3$. Suppose that the minimum is attained in the case $A_1+A_4$. 

\begin{lemma}
\label{area1}
$\area\big(s(B)\setminus
MNKL\big)+\frac{1}{2}(\deff_{(0,1)}(B)^2+\deff_{(1,1)}(B)^2)\geq \frac{3}{8}a^2$.
\end{lemma}
\begin{proof} The area of $s(B)\setminus
MNKL$ is at least $A_1+A_4$, $\deff_{(0,1)}(B)=x,\deff_{(1,1)}(B)=y$, and  $$A_1+A_4+\frac{1}{2}x^2+\frac{1}{2}y^2=\frac{1}{2}(a^2-ab+b^2+xb+y(a-b-x)+x_1(b-y)+y_1(a-b-x)).$$
Minimizing, we get $x_1=y_1=0$. Next, $y=0,x=0$. Finally, minimizing
$\frac{1}{2}(a^2-ab+b^2)$ with respect to $b$, we
obtain $\frac{3}{8}a^2$.

\end{proof}

\begin{proof}[Proof of Lemma \ref{area4}]
Using the previous Lemma, we get 
$$\area(\conv(B))+\frac{1}{2}\sum_{u\in\dirr}\deff_u(B)^2\geq 
\frac{1}{2}(m-a)^2+a(m-a)+\frac{3}{8}a^2\geq
\frac{3}{8}m^2,$$
 and equality is attained if $a=m$.
\end{proof}

\begin{cor}
\label{minimalw}
As a side effect, for the special case $a=m$ Lemma \ref{area4} gives
\end{cor}

\begin{th_old}[\cite{thinnest}, based on \cite{area}, p. 716, formula $II_3$, p. 715
formula $I$] Let $B\subset \ZZ^2$ be a finite set. Then  $\area(\conv(B))\geq\frac{3}{8}\omega(B)^2$.
\end{th_old}

In fact, from the above proofs it is easy to extract the extremal cases and exact bounds: if $\omega(B)=2k$, then $\area(\conv(B))\geq
\frac{3}{2}k^2$, and if $\omega(B)=2k+1$, then $\area(\conv(B))\geq
\frac{1}{2}(3k^2+3k+1)$.

\begin{rem}
The best constant $c_n$ in the inequality $\mathrm{volume}(\conv(B))\geq
c_n\omega(B)^n$ for $B\subset \ZZ^n$ is not known for $n>2$. The above theorem
says that $c_2=\frac{3}{8}$.
\end{rem}

\section{The proofs of the Exertion Theorems}
Firstly, we introduce the notation which we use throughout the remainder of this paper. Then we prove Lemma~\ref{lemma_implication} and the
Exertion Theorems.

\subsection{Notation}
\label{sec_notation}
Let $H$ be a tropical curve, given by \eqref{eq_tropical}. The extended Newton
polyhedron of $H$ is
$\widetilde\A$. We suppose that the point $P\in H$ is not a vertex of $H$.  We assume that  $P=(0,0)$ and the
edge $E$ containing $P$ is horizontal. We consider the long edge
$\mathfrak{E}=E_P((1,0))$. 

Call the vertices on $\mathfrak{E}$ from left to right
$A_1,A_2,A_3,\dots,A_n$. Clearly, we have $\I_P((1,0))=\bigcup_{i=1}^n \{A_i\}$ (Def.~\ref{def_I}). We denote by $E_i$ the edge of $H$
such that $E_i\subset \mathfrak{E}$ and the left end of $E_i$, if it exists, is the
point $A_i$. If $\mathfrak{E}$ contains an infinite edge of $H$ without a left
end, as in Example \ref{secondexample}, we call it $E_0$.  Let
$P$ belong to $E_\ell$. Refer to Figure \ref{extNewton} for this notation.

If $A_1$ is the left end of
$\mathfrak{E}$, then for the consistency of notation we add a ``fictive'' edge $E_0$ which has length
zero; $d(E_0)$ will denote the leftmost vertex of the face
$d(A_1)$. We say that $d(E_0)$ is a vertical edge of zero length.  Similarly, if $A_n$ is the right end of
$\mathfrak{E}$, then we add a ``fictive'' edge $E_{n}$ which has length
zero;
$d(E_{n})$ will denote the rightmost vertex of the face $d(A_n)$. Now,
regardless of finiteness of $\mathfrak{E}$, we always have edges
$E_0,E_1,\dots,E_{n}$. Since $\E$ is horizontal, it follows from
Proposition \ref{trop_prop} that for each $i=0,\dots,n$ the edge
$d(E_i)$ is vertical.

\begin{defi}
Refer to Figure~\ref{graph}(A). Let $x_i$ be the $x$-coordinate of the edge $d(E_i)$. By $y_i \leq y^i$ we denote the $y$-coordinates
of the endpoints of $d(E_i)$, and by $m_i=y^i-y_i$ the lattice length
of $d(E_i)$.  
\end{defi}

Note that we have $y_i=y^i$ if and only if $i=0$ (resp. $i=n$) and the long edge
$\E$ is finite on the left (resp. right) side.

\begin{prop} For each $i=1,\dots,n$, we have 
\begin{equation}
\label{eq_area}
\area(d(A_{i}))\geq \frac{1}{2}(x_{i}-x_{i-1})(m_{i}+m_{i-1}).
\end{equation}
\end{prop}
\begin{proof}
Since $d(A_{i})$ has two vertical sides of lengths $m_i,m_{i-1}$, the
inequality follows from the convexity of $d(A_{i})$.
\end{proof}

\begin{defi}
Recall that for each edge $E'$ of $H$, there is the dual edge $d(E')$ in the
subdivision of $\Delta$. Also, all the edges in the subdivision of
$\Delta$ arise as the projections of the edges of
$\widetilde\A$. We denote by $L(d(E'))$ the lifting of an edge $d(E')$ in the
boundary of $\widetilde{\A}$.
\end{defi}

If $d(E_0)$ is a point, then we denote by $L(d(E_0))$ the corresponding vertex of $\widetilde{\A}$. We apply the same rule for $d(E_n)$: look at the point $d(E_4)$ in Figure~\ref{extNewton}.

\begin{prop}
\label{horison2}
For each $i=1,\dots,n$, the face of $\widetilde{\A}$ spanned by
$L(d(E_{i-1}))$ and $L(d(E_i))$ projects to the face $d(A_i)$.
\end{prop}
\begin{proof}
This follows from Proposition \ref{trop_prop}. Refer to Figure \ref{extNewton}.
\end{proof}

\begin{figure} [htbp]
\begin{center}

\begin{tikzpicture}
[y= {(0.2cm,1cm)}, z={(0cm,0.5cm)}, x={(3cm,0cm)},scale=0.3]

\draw (-2,4,4)--(-0.5,5,2)--(0,4,8);
\draw (1,4,3)--(-0.5,5,2);

\filldraw[black,fill=gray!20](-2,2,4)--(0,-0.5,4)--(0.444,-0.333,4)--(0.8,1.8,4)-- (0.8,4,4)--(-0.333,4.666,4)--(-2,4,4)--cycle;

\draw[very thick] (-2.5,1,2)--(-2.5,2,2);
\draw[very thick] (-2,2,4)--(-2,4,4);
\draw[very thick] (0,1,8)--(0,4,8);
\draw[very thick] (1,2,3)--(1,4,3);
\draw[very thick] (2,2,-6)--(2,2,-6);

\draw (-3.3,-0.1,6) node{$L(d(E_{0}))$};
\draw (2.7,3,-6) node{$L(d(E_{4}))$};

\newcommand{\rr}{-14}

\filldraw[black,fill=gray!20](-2,2,\rr)--(0,-0.5,\rr)--(0.444,-0.333,\rr)--(0.8,1.8,\rr)-- (0.8,4,\rr)--(-0.333,4.666,\rr)--(-2,4,\rr)--cycle;
\draw[very thick] (-2.5,1,\rr)--(-2.5,2,\rr);
\draw[very thick] (-2,2,\rr)--(-2,4,\rr);
\draw[very thick] (0,1,\rr)--(0,4,\rr);
\draw[very thick] (1,2,\rr)--(1,4,\rr);
\draw[very thick] (2,2,\rr)--(2,2,\rr);
\draw (-3.25,-1,\rr)--(2.8,-1,\rr)--(2.8,5,\rr)--(-3.25,5,\rr)--cycle;
\draw (-2.5,1,\rr)--(-2,2,\rr)--(0,1,\rr)--(1,2,\rr)-- (2,2,\rr);
\draw (-2.5,2,\rr)--(-2,4,\rr)--(0,4,\rr)--(1,4,\rr)--(2,2,\rr);
\draw (2,2,\rr) node {$\bullet$};
\draw (-2.5,0,\rr) node{$d(E_0)$};
\draw (-1.75,1,\rr) node{$d(E_1)$};
\draw (0,0,\rr) node{$d(E_2)$};
\draw (1,1,\rr) node{$d(E_3)$};
\draw (2.3,1,\rr) node{$d(E_4)$};

\draw (2,2,-6)--(1,2,3)--(0,1,8)--(-2,2,4)--(-2.5,1,2);
\draw(2,2,-6)--(1,4,3)--(0,4,8)--(-2,4,4)--(-2.5,2,2);
\draw (0,-2,0)--(0,1,8);
\draw (0,-2,0)--(-2,2,4);
\draw (0,-2,0)--(-2.5,1,2);
\draw (1,-2,-1)--(0,-2,0);
\draw (1,-2,-1)--(1,2,3);
\draw (1,-2,-1)--(2,2,-6);
\draw (1,-2,-1)--(0,1,8);
\draw (2,7,-6)--(1,7,3)--(0,7,8)--(-2,7,4)--(-2.5,7,2);
\draw[dashed] (-2,7,4)--(0.8,7,4);
\draw(2,7,3) node{$z=g(x)$};
\draw (2.3,2.5,3) node{$L(d(E_{3}))$};
\draw (0,0,7) node{$L(d(E_2))$};
\draw (-2,4,6) node{$L(d(E_{1}))$};
\end{tikzpicture}
\begin{tikzpicture}
\draw (0,0) node{$\bullet$};
\draw (1.5,0) node{$\bullet$};
\draw (0.5,0) node{$\bullet$};
\draw (-1,0) node{$\bullet$};
\draw (-2,0) node{$\bullet$};
\draw (0,0) node[below]{$P$};
\draw (1.5,0) node[above]{$A_4$};
\draw (0.5,0) node[above right]{$A_3$};
\draw (-1,0) node[above left]{$A_2$};
\draw (-2,0) node[above left]{$A_1$};
\draw (-2.5,0) node[below]{$E_0$};
\draw (-1.5,0) node[below]{$E_1$};
\draw (-0.5,0) node[below]{$E_2$};
\draw (1,0) node[below]{$E_3$};

\draw [thick] (1.5,-0.5)--(1.5,0)--(2.5,0.5);
\draw [thick] (1.25,-0.75)--(0.5,0)--(0.5,1);
\draw [thick] (-1.25,-0.5)--(-1,0)--(-1,1);
\draw [thick] (-1.75,-0.5)--(-2,0)--(-2.25,1);
\draw [thick] (-3,0)--(1.5,0);
\end{tikzpicture}

\caption {On the left we see a part of the extended Newton polyhedron, which corresponds
  to a horizontal long
  edge on the right. The long edge $E_P((1,0))$ consists of the edges
  $E_0,E_1,E_2,E_3$, $l=2$, and $\I(P)=\{A_1,A_2,A_3,A_4\}$. The edges $L(d(E_i))$ of $\widetilde \A$ are
  depicted as thick black horizontal intervals, while a section of the extended
  Newton polyhedron by a horizontal plane is marked in gray, as well as its projection onto the $xy$-plane. The
  projection of $\widetilde\A$ onto the $xz$-plane is also depicted; the projection of
  the section is dashed. Note that we added a fictive edge $E_4$, and $d(E_4)$ is the rightmost vertex of $d(A_4)$. }
\label{extNewton}
\end{center}
\end{figure}
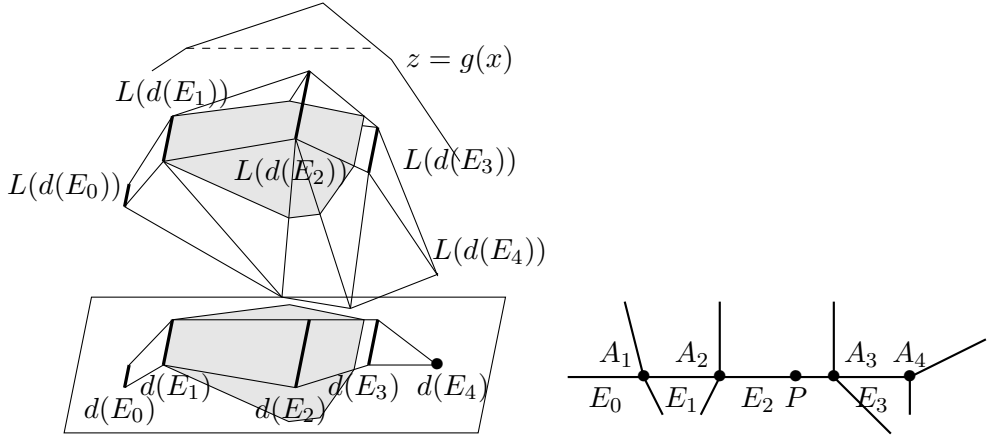

The edge $E_\ell$ is horizontal and passes through $(0,0)$. That
implies the following lemma. Nevertheless, we give more details to
illustrate the notation. 

\begin{lemma} The direction of the edge $L(d(E_\ell))$ is $(0,1,0)$ and $L(d(E_\ell))$ is
 higher than all other points of $\widetilde{\A}$. 
\end{lemma}

\begin{proof} Refer to Figure \ref{example3}. The top end $(x_l,y^l)\in\A$
of $d(E_l)$ represents the tropical monomial $M_1=\val(a_{x_ly^l})+x_lX+y^lY$ of
$\Trop(F)$; $M_1$ dominates other
monomials in the region above the edge $E_l$. The bottom end
$(x_l,y_l)\in\A$ of $d(E_l)$ represents the monomial $M_2=\val(a_{x_ly_l})+x_lX+y_lY$ which dominates other
monomials in the region below the edge $E_l$. Therefore
$M_1$ and $M_2$ are equal on the
edge $E_l$, in particular at the point $(0,0)$; therefore
$\val(a_{x_ly^l})=\val(a_{x_ly_l})$, hence $L(d(E))$ is
horizontal. Furthermore, $\max_{(i,j)\in\A} (\val(a_{ij})+iX+jY) =
\val(a_{x_ly^l})=\val(a_{x_ly_l})$ at the point $(0,0)$. 

If for some $i,j$ we have $\val(a_{ij})=\val(a_{x_ly^l})$, then $i=x_l$, otherwise
$P=(0,0)$ is a vertex of $H$. It follows from the maximality of
$\val(a_{x_ly^l})+x_lX+y^lY$ in the region above $E_l$ that $j\leq y^l$; then $y_l\leq j$
by symmetric reasoning.
\end{proof}

Refer to Figure \ref{extNewton}: the height of each bold edge $d(E_k)$ on the left side of the picture is greater then the heights
$\val(a_{ij})$ of the points $(i,j,\val(a_{ij}))$ such that $(i,j)$ lies to the left of $E_k$. In other
words, the projections of the bolded edges on the $xz$-plane lie on the
boundary of the $xz$-projection of $\widetilde{\A}$.

\begin{lemma}
\label{structure}
Consider an edge $E_q$ with $q<l$. For each $(i,j)\in \A$ with the
property 1) $i<x_q$ or 2) $i=x_q, j<y_q$, or 3) $i=x_q,j>y^q$, the number
$\val(a_{ij})$ is less than $\val(a_{x_qy_q})=\val(a_{x_qy^q})$. The
symmetric statement holds for $q>l$.
\end{lemma}

\begin{proof} Refer to Figure \ref{extNewton}. Each two consecutive edges
$d(E_i),d(E_{i+1})$ bound the face $d(A_{i+1})$, therefore the edges
$L(d(E_i)),L(d(E_{i+1}))$ (bolded in Figure \ref{extNewton}) also bound a face of the polyhedron $\widetilde{\A}$. The edges $d(E_i)$ are all
parallel to $d(E_l)$, therefore all the edges $L(d(E_i))$ are
parallel to each other as well. 
Provided $\widetilde{\A}$ is a convex polytope, all the points $(i,j,\val(a_{ij}))$ lie
under each plane passing through a face of $\widetilde{\A}$. The part with $q>l$ can be
proven by a word-by-word repetition of the above arguments. 
\end{proof}

\begin{defi}
Define $v_q:= \val(a_{x_qy_q})=\val(a_{x_qy^q})$, the height of the edge $L(d(E_q))$.
\end{defi}

Lemma \ref{structure} implies that $v_0<v_1<\dots<v_\ell>v_{l+1}>\dots>v_n$.

Let us project the boundary of $\widetilde{\A}$ to the
$xz$-plane. Each edge $L(d(E_i))$ is projected to the point
$(x_i,v_i)$ (Figure \ref{extNewton}(A)  and Figure
\ref{graph}(B) show examples of the result of such a projection). 
\begin{defi}
Let $g(x)$ equal $\max \{z|(x,y,z)\in \widetilde{\A}\}$. 
\end{defi}

The $xz$-projection of the face of $\widetilde\A$ stretched on the edges
$L(d(E_i)),L(d(E_{i+1}))$ coincides with the graph of $g$ on the
interval $[x_i,x_{i+1}]$, i.e., with the interval
$(x_i,v_i),(x_{i+1},v_{i+1})$ (compare Figures~\ref{extNewton} and
\ref{graph}).

For $x'\in [x_0,x_n]$ let $\hat{g}(x')$ be the length of the interval
excised from the line $z=g(x')$ by the graph
of $g$ (see Figure \ref{graph}, and the definition before Lemma \ref{function}).

\begin{rem}
If $P$ is a vertex of $H$, then we can repeat all the above steps for each long edge through $P$. 
\end{rem}

\subsection{The proof of Lemma \ref{lemma_implication}}

In Example \ref{secondexample}, $G$ can be written as
$$t^{-3}\cdot x(y-1)^3+t^{-2}\cdot x(x-1)(y-1)^2+t^{-1}\cdot(x-1)^2(y-1)+t^2\cdot(x-1)^3.$$ Therefore, in that example
the extended Newton polyhedron is made of layers of $m$-thick sets, namely
$\mathrm{supp}(x(y-1)^3),\mathrm{supp}(x(x-1)(y-1)^2),\mathrm{supp}((x-1)^2(y-1)),\mathrm{supp}((x-1)^3)$. 

Let $H=\Trop(C)$ and $\mu_{(1,1)}(C)\geq m$. We will prove that the horizontal sections of $\widetilde\A$ passing through the edges $L(d(E_i)), i = 0,\dots,n$ are $m$-thick. Then we extend this result to all the horizontal sections by Proposition~\ref{concavity}.

\begin{prop}
\label{prop_maximal}
If $P$ is not a vertex of $H$, then the edge $d(E_l)$ (see Section \ref{sec_notation} for the notation) has the lattice length at least $m$.
\end{prop}
\begin{proof} Let $\mu'=\max\{\mu\in \RR|\A_\mu\ne\varnothing\}$.
Clearly, $d(E_l) = \conv(\A_{\mu'})$. By the \hyperref[thickness_lemma]{$m$-thickness Lemma}, $d(E_l)$ is $m$-thick, which finishes the proof.
\end{proof}

\begin{rem}
\label{rem_contain}
If $P$ is a vertex of $\Trop(C)$, then the same reasoning shows that $\widetilde\A_{\mu'}=d(P)$
is $m$-thick. Furthermore, $\widetilde\A_\mu$ (Def.~\ref{def_wamu}) always contains $\A_{\mu'}$ for each $\mu<\mu'$.
\end{rem}

\begin{figure} [htbp]

\begin{tikzpicture}[scale=0.7]
\draw (0,0) node[below]{$x_0$};
\draw (-0.1,3.45) node{$m_0=0$};
\draw (-0.5,3) node{$y_0$};
\draw[dashed](-0.2,3)--(0,3);
\draw[dashed](0,0)--(0,3);

\draw (1,0) node[below]{$x_1$};
\draw[thick] (1,1.4)--(1,3.4);
\draw[dashed](1,1.4)--(1,0);
\draw (1,2.4) node{$m_1$};
\draw (2,0) node[below]{$x_2$};
\draw[thick] (2,1)--(2,4);
\draw (-0.5,1) node{$y_2$};
\draw[dashed](-0.3,1)--(2,1);
\draw[dashed](-0.3,4)--(2,4);
\draw (-0.5,4) node{$y^2$};
\draw (2,1) node[below]{$(x_2,y_2)$};
\draw (2,4) node[above]{$(x_2,y^2)$};
\draw (2,3) node{$m_2$};
\draw (4,0) node[below]{$x_3$};
\draw[thick] (4,0.5)--(4,4);
\draw (4,2) node{$m_3$};
\draw (7,0) node[below]{$x_4$};
\draw[thick] (7,1)--(7,3.7);
\draw (7,2) node{$m_4$};
\draw (8,0) node[below]{$x_5$};
\draw[thick] (8,0.8)--(8,2.7);
\draw (8,1.5) node{$m_5$};
\draw (9,0) node[below]{$x_6$};
\draw[thick] (9,2)--(9,2.5);
\draw (8.9,2.2) node[right]{$m_6$};
\draw(0,3)--(1,3.4)--(2,4)--(4,4)--(7,3.7)--(8,2.7)--(9,2.5)--(9,2)--(8,0.8)--(7,1)--(4,0.5)--(2,1)--(1,1.4)--cycle;
\draw(4,-1) node{$(A)$};
\end{tikzpicture}
\qquad
\begin{tikzpicture}[scale=0.7]
\draw (0,0) node[below]{$x_0$};
\draw (1,0) node[below]{$x_1$};
\draw (2,0) node[below]{$x_2$};
\draw (3,0) node[above]{$b$};
\draw (3,0) node {$\bullet$}; 
\draw (3,3.25) node{$\bullet$}; 
\draw (3,3.25) node[above]{$g(b)$};
\draw (4.15,0) node[below]{$x_3$};
\draw (7,0) node[below]{$x_4$};
\draw (8,0) node[below]{$x_5$};
\draw (9,0) node[below]{$x_6$};
\draw (-0.1,0)--(9.1,0);
\draw(0,0)--++(1,2)--++(1,1)--++(2,0.5)--++(3,-0.8)--++(1,-0.8)--++(1,-1.4);
\draw [thin](1,0)--(1,2);
\draw (1,1) node{$m_1$};
\draw (0.9,1.6) node[left]{$g(a)$};
\draw [thin](2,0)--(2,3);
\draw (2,1) node{$m_2$};
\draw (7,1) node{$m_4$};
\draw [thin](7,0)--(7,2.7);
\draw (8,1) node{$m_5$};
\draw [thin](8,0)--(8,1.9);
\draw [thin](9,0)--(9,0.5);
\draw(3,3.25)--++(6,0);
\draw(7,3.5) node {$z=g(b)$};
\draw [<->] (3,3)--(5,3);
\draw(4.6,2.5) node{$\hat{g}(b)$};

\draw (0.8,0) node[above]{$a$};
\draw (0.8,0) node {$\bullet$};
\draw [<->] (0.8,1.6)--(8.2,1.6);
\draw(4.6,1.1) node{$\hat{g}(a)$};
\draw(4,-1) node{$(B)$};
\end{tikzpicture}\qquad
\begin{center}
\caption {Projections of $\widetilde\A$ to the $xy$-plane (A) and to
  the $xz$-plane (B) are depicted. The number $x_i$ is the $x$-coordinate of the edge
  $d(E_i)$ in $(A)$. In this example, the long edge $E_P((1,0))$ is finite from the left side
  (therefore $m_0=0$) and infinite from the right side (therefore $m_n=m_6>0$). By definition $g(x_i)=v_i$ in $(B)$. Also,
  $\hat{g}(a)$ and $\hat{g}(b)$ are presented in $(B)$, and
  $\hat{g}(x_3)=0, l=3$. The key observation is
  that $\hat{g}(x_i)+m_i\geq m$ (Lemma \ref{hat}).
  Furthermore, $\hat{g}$ is concave on $[x_i,x_{i+1}]$ for each $i$; see Proposition
  \ref{concavity} for details.}
\label{graph}
\end{center}
\end{figure}
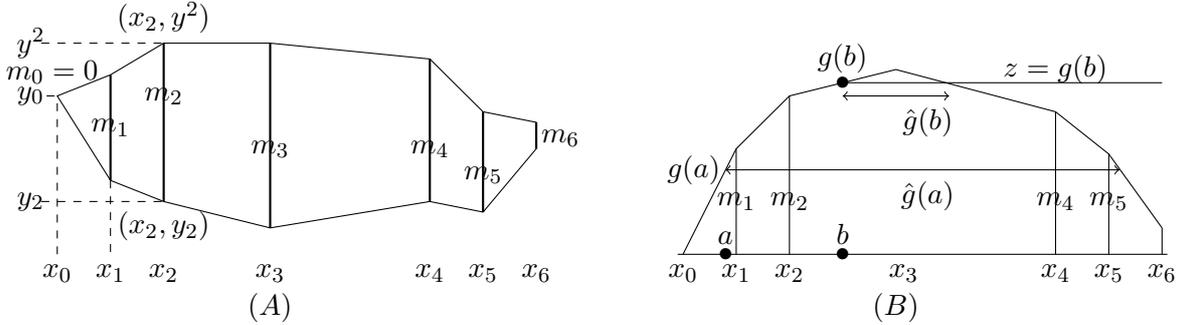

By the \hyperref[thickness_lemma]{$m$-thickness Lemma}, for each $i=0,\dots, n$, the set $\A_{v_i}$ is
$m$-thick. 
The following Lemma estimates the
length of $d(E_i)$ via the width $\hat{g}(x_i)$
of the horizontal section through $L(d(E_i))$.

\begin{lemma}
\label{hat}
For each $i=0,1,\dots,n$, the length $m_i$ of the edge $d(E_i)$ is at least $m-\hat{g}(x_i)$.
\end{lemma}
\begin{proof} We draw
the horizontal section $\{z=v_i\}$ through the bold edge $L(d(E_i))$; refer to Figure \ref{extNewton} where $i=l-1$. Consider the line $z=g(x_i)$ in the $xz$-plane. Suppose that
  the projection of the interval, excised on this line by the graph of $g$, onto the $x$-axis is
$[x_i,x_i'], x_i'>x_i$. In fact, the length $\hat{g}(x_i)$ of the dashed line in
Figure \ref{extNewton} satisfies $\hat{g}(x_i)=x_i'-x_i=
\omega_{(1,0)}(\widetilde{\A}\cap \{z=v_i\})=\omega_{(1,0)}(\widetilde \A_{v_i})$. The set
$\A_{v_i}$ is inside the strip $\{ (x,y)|\ x_i\leq x\leq x_i'\}$, and $\A_{v_i}$ is $m$-thick by the \hyperref[thickness_lemma]{$m$-thickness Lemma}.  Since $\conv(\A_{v_i})\cap \{x=x_i\}$ is $d(E_i)$,
this lemma follows from the definition of $m$-thickness.\end{proof}

\begin{rem}
In fact, $\A_{v_i}$ is contained in the $xy$-projection of $\{z=v_i\}\cap \widetilde{\A}$, but does not necessarily coincide with it. 
\end{rem}

Consider the following piecewise linear function $f$ on the interval $[x_0,x_n]$:
let $f(x_i)=m_i$ for  $i=0,\dots, n$, then extend $f$ to be linear on each
interval $[x_i,x_{i+1}]$. 

\begin{prop}
\label{prop_linear}The length of the left vertical side of $\conv(\widetilde\A_{g(x)}), x\leq x_l$ is at least $f(x)$.
\end{prop}

\begin{proof}
It follows from the fact that the face of $\widetilde\A$ stretched on $L(d(E_i)),L(d(E_{i+1}))$ contains the trapezoid stretched on $L(d(E_i)),L(d(E_{i+1}))$, and $f$ calculates the lengths of its intersection with horizontal sections. 
\end{proof}

\begin{lemma} 
\label{lemma_mthick}
The inequality $f(x)+\hat{g}(x)\geq m$ holds on the interval $[x_0,x_n]$.
\end{lemma}

\begin{proof} For each $i=0,\dots,n$  the inequality $f(x_i)+\hat{g}(x_i)\geq m$ is satisfied by Lemma \ref{hat}. Consider an interval $[x_i,x_{i+1}]$. Since $f$ is
linear and $\hat{g}$ is concave on $[x_i,x_{i+1}]$ (Proposition \ref{concavity}), we have $f(x)+\hat{g}(x)\geq m$ for
each $x\in[x_1,x_{i+1}]$.\end{proof}

\begin{proof}[Proof of Lemma \ref{lemma_implication}] Suppose that $P$ is not a
  vertex of $\Trop(C)$ and $P$ belongs to a horizontal edge of
  $\Trop(C)$. It follows from Remark \ref{rem_contain} that it is
  enough to check the $m$-thickness of
  $\widetilde\A_\mu$ only in the direction $(1,0)$. The latter follows
  from Lemma \ref{lemma_mthick} and Proposition \ref{prop_linear}.
If $P$ is a vertex of $\Trop(C)$, then, again, Remark
\ref{rem_contain} implies that we need to check the $m$-thickness of $\widetilde\A_\mu$ only in the directions of the edges through
$P$. For each edge through $P$, we use Propositions~\ref{tropic_action}, \ref{translation1} for making this edge horizontal. Then we repeat the above arguments.
\end{proof}

\subsection{Proof of the Exertion theorem for edges}

The second part of the Exertion Theorem for edges is proved in
Proposition \ref{prop_maximal}.
\begin{lemma} [c.f. Lemma \ref{prop_width}]
\label{boundary}
Refer to Figure~\ref{graph}(A) for the notation. If $H$ is admissible (Def.~\ref{admissible}) and $\mu^{\trop}_P(H)\geq m$, then $x_{n}-x_0\geq m$.
\end{lemma}
\begin{proof} Let us suppose that $x_{n}-x_0< m$. If $m_0,m_n>0$, then $\omega_{(1,0)}(\A)=x_n-x_0<m$, and the curve $H$ is not
admissible. If $m_0=0$ and  $m_n>0$, then $\omega_{(1,0)}(\A_{v_0})<m$
and $\A_{v_0}$ does not have two vertical
sides, which contradicts the fact that $\A_{v_0}$ is $m$-thick (Proposition~\ref{two_vertical}). If both $m_0=m_n=0$, then we apply the
above argument for $\A_{\max (v_0,v_n)}$.\end{proof}

\begin{prop}
\label{prop_implication}
If a point $P$ is of multiplicity at least $m$ in the intermediate sense,
then $P$ is of multiplicity at least $m$ in the intrinsic sense (Def.~\ref{def_intrinsic}).
\end{prop}

\begin{proof} Indeed, let us take a generalized tropical line $L$. We will verify
  Def.~\ref{def_intrinsic}. If $P$ is the vertex of $L$ or $TC(P)$
  does not contain the vertex of $L$, then the fact that $\widetilde\A_{\mu'}$
  is $m$-thick (Remark~\ref{rem_contain}) implies that $L\cdot_PH\geq m$. If the vertex $V$ of $L$
  belongs to a long edge through $P$, then we use the
  notation in Section~\ref{sec_notation}. We may assume that $L$ has a horizontal edge passing through $P$. Let $V$ belongs to $E_k$. Draw the horizontal section
  through $L(d(E_k))$. A direct calculation and Lemma \ref{hat} show that
  $m$-thickness of $\widetilde\A_{v_k}$ implies that $L\cdot_P H\geq m$. 
\end{proof}

It follows from Lemma \ref{boundary} that

\begin{prop} There are points $b,c\in [x_0,x_n]$
such that $c-b=m$ and one of the following statements hold
\begin{itemize}
\item  $g(b)=g(c)$,
\item $g(b)\leq g(c), c=x_n$,
\item $g(b)\geq g(c), b=x_0$.
\end{itemize}
\end{prop}

The points $b,c$ are chosen in such a way that $\hat{g}_{[b,c]}(x)=(\hat{g}_{[x_0,x_n]})|_{[b,c]}(x)$
for $x\in [b,c]$. By $h|_{[b,c]}$  we mean the restriction of $h$ to $[b,c]$. The definition of $f(x)$ is given before Lemma \ref{lemma_mthick}.
\begin{proof}[Proof of Theorem \ref{exertion_edges}] We complete the proof, applying Lemma \ref{function} on the
interval $[b,c]$ of length $m$:

$\area(\Inf(P))\geq \int_{x_0}^{x_n}f(x)dx \geq \int_b^c f(x)dx \geq \int_b^c(m-\hat{g}(x))dx\geq m(c-b)-\frac{(c-b)^2}{2}=\frac{m^2}{2}$.
\end{proof}

\begin{prop}
\label{prop_position}
If $E_P((1,0))$ coincides with the interval $[A_1,A_n]$ and $x_{n}-x_0=m$, then only one point $P\in [A_1,A_n]$ can be a point of multiplicity $m$ in the intermediate sense.
\end{prop}
\begin{proof}
Indeed, using the $m$-thickness property of $\A_{\max(v_0,v_n)}$, we conclude that $v_0=v_n$ (cf. Lemma \ref{boundary}). This is equivalent to the fact that $\val(a_{x_0y_0})=\val(a_{x_ny_n})$, where $(x_0,y_0)$ is the leftmost vertex of $d(A_1)$ and $(x_n,y_n)$ is the rightmost vertex of $d(A_n)$; see Figure~\ref{graph}. All this notation (Section~\ref{sec_notation}) was developed for the case $P=(0,0)$.   Then, using Proposition \ref{translation1}, we see that the choice of another point $P'\in[A_1,A_n]$ and a subsequent change of the coordinates in order to make $P'=(0,0)$ will destroy the equality $v_0=v_n$.
\end{proof}

We can prove in Example \ref{secondexample}, that if $P$ is of multiplicity $3$ in the extrinsic sense, then $P$ must divide the edge  in the ratio $1:2$. Also, in the hypothesis of the above proposition, it is possible to determine the position of the singular point via tropical modifications (\cite{guide}). 

\subsection{Proof of the Exertion theorem for vertices}
Now we are in the
hypothesis of the Exertion Theorem for vertices,  i.e.
$\mu^{\trop}_P(H)\geq m$, $P$ is a vertex of $H$, and the Newton polygon $\Delta$ of $H$ has minimal lattice width at least $m$.
For each direction $u\in \dirr$ such
that the face $d(P)$ has at most one side perpendicular to $u$, the width
$\omega_u(d(P))$ is at least $m$.
This follows from Lemma \ref{vertical_sides}, since
  $d(P)$ is $m$-thick. 
  
Suppose that the point $P$ belongs to an edge $E\subset H$ of direction
$u$. If $\omega_u (d(P))<m$, then the face $d(P)$ has two sides of
lattice length at least $\deff_u(d(P))$ (Def.~\ref{def_def}), and these sides
are perpendicular to the vector $u$; see Figure \ref{est4}.

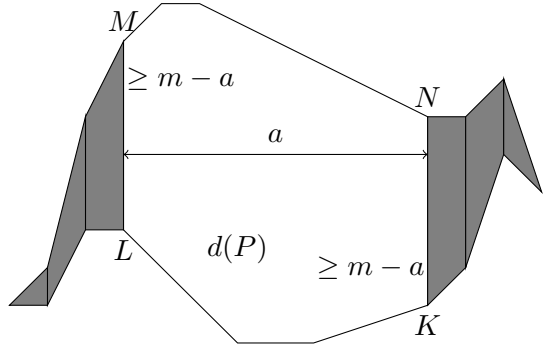
\begin{figure} [htbp]
\begin{center}
\begin{tikzpicture}[scale=0.5]

\draw (0,4)--(0,9)--(1,10)--(2,10)--(8,7)--(8,2)--(5,1)--(3,1)--cycle;
\draw [fill=gray](0,4)--(-1,4)--(-1,7)--(0,9)--cycle;
\draw [fill=gray](-1,4)--(-2,2)--(-2,3)--(-1,7)--cycle;
\draw [fill=gray](-2,2)--(-3,2)--(-2,3)--cycle;
\draw[<->](0,6)--(8,6);
\draw (4,6.5) node{$a$};
\draw (1.5, 8) node{$\geq m-a$};
\draw (6.5,3) node{$\geq m-a$};
\draw [fill=gray](8,7)--(9,7)--(9,3)--(8,2)--cycle;
\draw [fill=gray](9,7)--(10,8)--(10,6)--(9,3)--cycle;
\draw [fill=gray](10,8)--(11,5)--(10,6)--cycle;
\draw (3,3.5) node{$d(P)$};
\draw(0,4) node[below]{$L$};
\draw(0,9) node[above]{$M$};
\draw(8,7) node[above]{$N$};
\draw(8,2) node[below]{$K$};
\end{tikzpicture}
\qquad
\caption {An example of the dual picture to a horizontal long edge
  through $P$, if $P$ is a vertex of
  $H$. We have
  $\omega_{(1,0)}(d(P))=a$ and $\mu_P^{\trop}(H)\geq m$, therefore the
  lengths of $LM$ and $NK$ are at least $m-a$.
The set $\bigcup d(Q)$ for $Q\in \I_{(1,0)}(P), Q\ne P$ is colored.
Lemma~\ref{additional} states that the sum of
the areas of the colored faces is at least $\frac{1}{2}(m-a)^2$.}
\label{est4}
\end{center}
\end{figure}

\begin{lemma}
\label{additional} If $\mu^{\trop}_P(H)\geq m$, $P$ is a vertex of
$H$, and $u\in\dirr$, then
\begin{equation}
\label{eq_defect}
\sum\limits_{V\in
  \I_P(u),V\ne P}\area(d(V))\geq
\frac{1}{2}\deff_u(d(P))^2.
\end{equation}
\end{lemma}

\begin{proof} Applying a change of coordinates (Proposition
  \ref{tropic_action}), we may assume that  $u=(1,0)$. Let $\omega_u(d(P))=a$. The faces of the subdivision contributing
to \eqref{eq_defect} are
colored in Figure \ref{est4}. Now we consider the set $\{(i,j)\in \ZZ^2\}$ where $\val(a_{ij})$
is maximal. It contains the vertices of $d(P)$ and maybe some integer
points inside $d(P)$. As in the proof of the Exertion Theorem for
  edges, we consider the sets $A_\mu$ for different $\mu$, and repeat
all the other steps. In the final step of the proof, instead of the
integral $\int_b^c(m-\hat{g})dx$ we consider the
integral $\int_{b}^{x_i}(m-\hat{g}) + \int_{x_{i+1}}^{c}(m-\hat{g})$ where
$x_i,x_{i+1}$ are the $x$-coordinates of the vertical sides of
$d(P)$. Finally, \begin{align*}\sum\limits_{Q\in \I_u(P),Q\ne
  P}&\area(d(Q))\geq \int_{b}^{x_i}(m-\hat{g})dx +
\int_{x_{i+1}}^{c}(m-\hat{g})dx\\&=m(x_i-b)+m(c-x_{i+1}) -
\int_{b}^{x_i}\hat{g}dx-\int_{x_i}^c\hat{g}dx\\
&=m(m-a)-\left(\frac{1}{2}(c-b)^2-\frac{1}{2}(x_{i+1}-x_i)^2\right) =
\frac{1}{2}(m-a)^2 = \frac{1}{2}\deff_u(d(P))^2,\end{align*} by 
Corollary \ref{cor_integral}.\end{proof}

\begin{proof}[Proof of Theorem \ref{exertion_vertices}] Indeed, it
  follows from Lemma \ref{additional} that $$\area^*(\Inf(P)) = \sum\limits_{\substack {u\in\dirr, \\ V\in
  \I_u(P),V\ne P}}\area(d(V)) + \area(d(P))\geq \area(d(P))+\frac{1}{2}\sum\limits_{u\in\dirr}\deff_u(d(P))^2,$$ 
and the latter expression is at least $\frac{3}{8}m^2$ by Lemma
\ref{area4}.

Similarly, by Lemma \ref{area3} we get
$$\area(\Inf(P)) \geq 2\cdot\area(d(P))+\frac{1}{2}\sum\limits_{u\in\dirr}\deff_u(d(P))^2\geq
\frac{1}{2}m^2.\qedhere$$
\end{proof}

\section{Discussion}
\label{discussion}
{\rightline {``The forceps of our minds are clumsy forceps,} \rightline{and crush
    the truth a little in taking hold of it.''} \rightline{H. G. Wells}}

In this section we show that a point of multiplicity $m$ can
impose fewer than $\frac{m(m+1)}{2}$ linearly independent conditions on the coefficients of the
equation of a curve. Also, we summarize what is known about tropical points of
multiplicity $m$.

\subsection{Examples and the Euler derivative}
\begin{ex}
\label{triangle}
Fix $k\in \mathbb N$. The polygon $T_k$ of the minimal area with $\omega(T_k)=2k$  is the
triangle with vertices $(0,0),(k,2k),(2k,k)$ (see Remark \ref{minimalw}). The triangle $T_k$ comes as the support set of the polynomial
$(1-3xy+xy^2+x^2y)^k=0$ which defines a curve $C$ with $\mu_{(1,1)}(C)=2k$. The area of $T_k$ is $\frac{3}{8}(2k)^2$, which shows
that the estimate in the \hyperref[exertion_vertices]{Exertion Theorem for vertices} is sharp.
\end{ex}

If $char(\KK)=0$, then $\mu_{(1,1)}(C) \geq 2k$ is
equivalent to the set of linear equations 
$\frac{\partial^{q+r}}{\partial^qx\partial^ry} F(x,y)=0, q+r< 2k$ in the coefficients of
the polynomial $F=\sum\limits_{(i,j)\in T_k}a_{ij}x^iy^j$. Note
that among these equations, there are at
least $$\frac{2k(2k+1)-(3k^2+3k+2)}{2}=\frac{k^2-k-2}{2}$$ linearly
dependent ones. Here $\frac{2k(2k+1)}{2}$ is the number of
equations and $\frac{3k^2+3k+2}{2}$ is the number of variables,
i.e., the number of integer points in $T$.  
\begin{ex}
\label{ex_freedom}
To see one more
phenomenon we consider the set
$$\A=\conv((0,0),(1,3),(6,3),(6,4),(3,6),(3,1))=T_3\cup \{(1,3),(3,1),(6,4)\}.$$ 
The only curve $C$ with support in
$\A$ and $\mu_{(1,1)}(C)=6$ is given by the equation
$(1-3xy+xy^2+x^2y)^3=0$. Hence adding three new monomials
$a_{13}xy^3+a_{31}x^3y+a_{64}x^6y^4$ does not add new degrees of
freedom and $a_{13},a_{31},a_{64}$ are always 0.
\end{ex}

We give the following explanation. Consider the constraint on
$a_{ij}$ imposed by the fact that $F_{xx}(1,1)=0$. That is $\sum
i(i-1)a_{ij}=0$. Note that the set of $a_{ij}$ with non-zero
coefficients in this equation is parametrized by $\A\setminus \{(i,j)|i(i-1)=0\}$. So, we say that $i(i-1)$ corresponds to $F_{xx}$.

In a similar way, given $\mu_{(1,1)}=6$, by considering linear combinations of
$F,F_x, F_{xy},\dots, F_{yyyyy}$, we can obtain all the polynomials in $i,j$
of degree at most five.  Next, $(6,4)$ is the only point in $\A$ where
$f(i,j)=(j-3)(i-j)(i-3)(i^2+j^2-ij-3j-3i+6)$ is not zero. The linear equation
corresponding to $f(i,j)$,

\begin{align*}(F_{xxxxy}-2F_{xxxyy}+2F_{xxyyy}&-F_{xyyyy}-3F_{xxxx}+4F_{xxxy}-\\
 &-4F_{xyyy}+3F_{yyyy}-12F_{xx}+12F_{yy}+24F_x-24F_y)|_{(1,1)}=0,
\end{align*} 
 written in terms of $a_{ij}$,  is just $a_{64}=0$. Similar
combinations of derivatives can be found for $a_{13}$ and $a_{31}$.

Let $\mathrm{char}
\ \KK=0$.
In this case, \cite{hyper2} contains the complete description of the matroid $M$ associated with the linear
conditions imposed by the $m$-fold point at $(1,1)$. Namely, all the dependent sets of $M$,
minimal by inclusion, are the sets of the type $\A\setminus\{(i,j)|G(i,j)=0\}$,
where $G\in\KK[i,j]$ is a polynomial of degree at most $m-1$.

Let $\A_G$ be $\A\setminus\{(t,w)|G(t,w)=0\}$. We call the operation $$\partial_G:\sum_{(i,j)\in \A} a_{ij}x^iy^j\to
\sum_{(i,j)\in \A_G} a_{ij}x^iy^j$$ {\it the
Euler derivative with respect to $G$}. Suppose that a tropical curve  $H$ is given by $\Trop(F)$ where $F$ is as in \eqref{eq_curve}.

\begin{prop}[\cite{hyper2}]\label{prop_method}
A point $P\in H$ is a point of multiplicity at least $m$ in the $\KK$-extrinsic
sense (Def.~\ref{def_extrinsic}) if and only if for each polynomial $G\in
\KK[i,j]$ of degree no more than $m-1$, the tropical curve given by
$\Trop(\partial_GF)$ passes through $P$.
\end{prop}

\begin{rem}
\label{euler} If $\mathrm{char}\ \KK=0$, then the above proposition 
implies the $m$-thickness property for
$\A$ if $\mu_{(1,1)}(C)=m$ (cf. Corollary \ref{cor_mthick}
). Indeed, if the set $\A$ is not $m$-thick, then there exists a
collection of $m-1$ lines $l_1,\dots, l_{m-1}$ such that $\A\setminus
\bigcup\{l_i\}=(i',j')\in\ZZ^2$. Let the polynomial $G$ be the product of the equations of the lines $l_i$. Clearly, $\mathrm{deg}(G)=m-1$. Then, $\partial_G
F=a_{i'j'}x^{i'}y^{j'}$, and $\Trop(\partial_G
F)$ is smooth at $P$. This contradicts to
Proposition \ref{prop_method}.  
\end{rem}

One can argue that in Examples \ref{triangle}, \ref{ex_freedom} we have a smaler
degree of freedom because the curves were reducible, so, look at the following example.
\begin{ex}
Consider the curve $C'$ given by the equation
$(x^2y+xy^2-3xy+1)^8+xy^4(x-1)^8=0$. It is irreducible,
$\mu_{(1,1)}(C')=8$ and the number of integer points in the Newton
polygon of $C'$ is 35, which is less than the number of linear
conditions, namely 36.
\end{ex}

\subsection{Tropical points of multiplicity $m$}
\label{tropical_multiplicity}
The aim of the present work was to improve the understanding 
of the combinatorics of tropical singular points. Applications of the Exertion Theorems for Nagata's conjecture can be found in \cite{2013arXiv1310.6684K}.

For a tropical curve $H$, if a point $P$ is of multiplicity at least
$m$ in the $\KK$-extrinsic sense (Def.~\ref{def_extrinsic}),
then $P$ is of multiplicity at least $m$ in the intermediate sense
(Def.~\ref{def_intermediate}); see Lemma \ref{lemma_implication}.

{\bf Question}: is it true that for each $m$-thick (Def.~\ref{mthick}) set
$B\subset \ZZ^2$, there exists a polynomial $G\in\QQ[x,y]$ defining the curve
$C'$ such that $\mu_{(1,1)}(C')\geq m$ and
$\conv(\mathrm{supp}(G))=\conv(B)$? 
As it is shown in Example~\ref{ex_freedom}, the answer is ``no''.

We say that a tropical curve $H$ can be lifted over a field $\KK$ if there exists a curve $C'$ over $\KK$ such that $\Trop(C')=H$.
Let a point $P\in H$ be of multiplicity $m$ in the $\KK$-extrinsic sense for some valuation field $\KK$. Suppose that $H$ can be lifted over another field $\KK'$ of the same characteristic. 

{\bf Question}: is it true that the point $P$ is of multiplicity $m$ in the $\KK'$-extrinsic sense? As far as the author knows, this is an open problem (though not very difficult).  

For a tropical curve $H$, if a point $P\in H$ is of multiplicity at least
$m$ in the intermediate sense,
then $P$ is of multiplicity at least $m$ in the intrinsic sense
(Def.~\ref{def_intrinsic}); see Proposition \ref{prop_implication}.

Note that the method in Proposition \ref{prop_method}, which allows us to verify the definition in the extrinsic sense, requires
information about all the valuations of the coefficients of the equation of the tropical curve $H$. Therefore, we have to know even
those coefficients which can be perturbed without changing $H$. Hence,
given only a tropical curve $H$, the verification of Def.~\ref{def_extrinsic}  is not straightforward.

On the other hand, it is enough to know only the dual subdivision of
the Newton polygon for $H$ in order to verify the definition in the intrinsic sense (Def.~\ref{def_intrinsic}). The
multiplicity in the intrinsic sense of a point $P\in H$ remains the same if we change the
lengths of the edges of $H$. 
Quite the contrary, for Def.~\ref{def_intermediate} of multiplicity in the intermediate sense, the lengths of the edges of $H$ are important because we operate with the extended Newton polyhedron $\widetilde \A$; see also Remark \ref{remark_passing}.

So, if a point $P$ is a
point of multiplicity $m$ in the extrinsic sense, then $P$ satisfies some necessary conditions, for example, estimates in the Exertion Theorems hold and can be easily verified. Nevertheless an ambiguity remains: it is possible that a lot of the points on an edge $E$ are of multiplicity $m$ in the extrinsic sense, but we cannot realize them as tropicalizations of $m$-fold points simultaneously; see examples in \cite{markwig, markwig2}. See also Proposition \ref{prop_position} for the case where we can prove that the position of $P$ is unique.

\bibliography{../../bibliography}
\bibliographystyle{abbrv}

\end{document}